\documentclass[11pt]{article}
\usepackage{amsmath}
\usepackage{amssymb}
\usepackage{mathtools}
\usepackage{amsthm}
\usepackage{graphicx}
\usepackage{amsfonts}
\newtheorem{thm}{Theorem}[section]
\newtheorem{lem}{Lemma}[section]

\newtheorem{cor}[thm]{Corollary}

\newtheorem{Q}{Question}

\newtheorem{Def}{Definition}[section]
\newtheorem{prop}{Proposition}[section]
\newtheorem{claim}{Claim}[section]

\def\N{\mathcal{N}}
\def\P{\mathcal{P}}
\def\Q{\mathcal{Q}}
\def\F{\mathcal{F}}

\def\B{\mathcal{B}}
\newcommand{\si}{\sigma}
\newcommand{\sm}{\setminus}
\newcommand{\rar}{\Rightarrow}
\newcommand{\nin}{\notin}

\numberwithin{equation}{section}

\begin{document}
\title{Neighborhood Complexes of Kneser Graphs, $KG_{3,k}$}
\author{ Nandini Nilakantan\footnote{Department of Mathematics and Statistics, IIT Kanpur, Kanpur-208016, India. nandini@iitk.ac.in.},  Anurag Singh\footnote{{Department of Mathematics and Statistics, IIT Kanpur, Kanpur-208016, India. anurags@iitk.ac.in.}}}

\date{\today}
\maketitle

\begin{abstract}

In this article, we prove that the neighborhood complex of the Kneser graph $KG_{3,k}$ is of the same homotopy type as that of a wedge of $\frac{(k+1)(k+3)(k+4)(k+6)}{4}+1$ spheres of dimension $k$. We construct a  maximal subgraph $S_{3,k}$ of $KG_{3,k}$, whose neighborhood complex deformation retracts onto the neighborhood complex of $SG_{3,k}$.
\end{abstract}

\noindent {\bf Keywords} : Hom complexes, Kneser Graphs, Discrete Morse theory.

\noindent 2000 {\it Mathematics Subject Classification} primary 05C15, secondary 57M15

\vspace{.1in}

\hrule


\section{{\bf Introduction}}
The Kneser Conjecture  (\cite{Lov78}), proved by Lov\'asz in 1978, states that the chromatic number of the Kneser graph $KG_{n,k}$ is $k+2$. Lov\'asz, in his proof, introduced a prodsimplicial complex corresponding to a graph $G$, called the neighborhood complex of $G$, denoted by $\N(G)$ and showed that $\N(KG_{n,k})$ is $(k-1)$-connected.  Schrijver \cite{Sch78} identified a vertex critical family of subgraphs of the  Kneser graphs called the stable Kneser graphs $SG_{n,k}$ and showed that the chromatic numbers of  both $SG_{n,k}$ and $KG_{n,k}$ are the same.
In 2003, Bj\"orner and de Longueville  (\cite{BL03}) proved that the neighborhood complex of $SG_{n,k}$ is homotopy equivalent to a $k$-sphere. 

In (\cite{Koz07}, Proposition 17.28), Kozlov proved that $\N(KG_{n,k})$ is homotopy equivalent to a wedge of spheres of dimension $k$.  In \cite{na}, we have shown that the homotopy type of  the neighborhood complex of the Kneser graph $KG_{2,k}$ is a wedge of $(k+4)(k+1)+1$ spheres of dimension $k$.

We consider a subgraph $S_{3,k}$ of $KG_{3,k}$, where the set of vertices of both the graphs is the same.  This graph is maximal in the sense that the addition of an extra edge to the graph $S_{3,k}$  changes the homotopy type of the corresponding neighborhood complex.  In this article, we prove:

 \begin{thm}\label{thm2} The neighborhood complex, $\N(S_{3,k})$, of $S_{3,k}$ collapses onto the neighborhood complex, $\N(SG_{3,k}),$ of the stable Kneser graph $SG_{3,k}$, for all $k \geq 0$.
  \end{thm}
  
  Since, $\N(SG_{3,k})$ is homotopic to $S^{k}$, the following is a consequence of Theorem \ref{thm2}.
  
  \begin{cor}\label{cor1.2}
The neighborhood complex, $\N(S_{3,k}),$ of $S_{3,k}$ is homotopic to the $k$-sphere $S^{k},$ for all $k \geq 0$.
  \end{cor}

The main result of this article is stated below:

 \begin{thm}\label{thm3}
 The neighborhood complex, $\N(KG_{3,k}),$ of the Kneser graph $KG_{3,k}$ is  homotopic to the wedge of $\frac{(k+1)(k+3)(k+4)(k+6)}{4}+1$ copies of the $k$-sphere $S^{k},$ for all $k \geq 0$.
  \end{thm}
 
\section{{\bf Preliminaries}}

An ordered pair $G=(V,E)$, where $V$ is the set of vertices and $E  \subseteq \{ e \subseteq V  \ | \ \# (e) =2 \}$ is the set of edges of $G$ is called a  {\it simple graph.}  An edge $e$ is denoted by $(v,w)$. All graphs in this article are assumed to be simple and undirected {\it i.e.}, $(v,w)=(w,v)$. The vertices $v, w \in V$ are called adjacent, if $(v, w)\in E$.

A collection of finite sets, where $\tau \in X$ and $\sigma \subset \tau$ implies $\sigma \in X$ is called a {\it finite abstract simplicial complex}. If $\sigma \in X$ and $\#(\sigma) =k+1$, then $\sigma$ is called a $k$-simplex.
A {\it prodsimplicial complex} is a polyhedral complex, each of whose cells is a direct product of simplices (\cite{Koz07}).
The {\it face poset} of a simplicial complex $X$ , $\F(X)$, is a poset whose elements are the nonempty faces of $X$ ordered by inclusion.

 The {\it neighborhood  of a vertex $v$} in a graph $G$, is defined as $N(v)=\{ w \in V(G) \ |  \
(v,w) \in E(G)\}$.  For $A\subset V(G)$, the neighborhood of $A$, $N(A)= \{x \in  V(G) \ | \ (x,a) \in E(G),
\,\,\forall\,\, a \in A \}$.
The {\it neighborhood complex}, $\N(G)$ of $G$ is an abstract simplicial complex with vertices being the non isolated vertices of $G$ and simplices being all those subsets of $V(G)$ with a common neighbor. We now introduce an important tool in topological combinatorics.

In a simplicial complex $X$, if $\tau, \sigma \in X$, $\sigma \subsetneq \tau$ and  $\tau$ is the only maximal simplex in $X$ that contains $\sigma$, then $\sigma$ is called a {\it free face} of $\tau$ and $(\sigma, \tau)$ is called a {\it collapsible pair}. The simplicial complex $Y$ obtained from $X$ by removing all those simplices $\gamma$  of $X$ such that $\sigma \subseteq \gamma \subseteq \tau$ is  called a  {\it (simplicial) collapse} of $X$. If there exist a sequence of collapses from $X$ to $Z$, then $X$ is said to collapse onto $Z$ and is denoted by $X \searrow Z$.

Let $n \geq 1$, $k \geq 0$ and $[n]$ be the set   $ \{1,2, \ldots ,n \}$.

\begin{Def}
The Kneser graph $KG_{3,k}$ is the graph where $V(KG_{3,k}) =\{ \{i,j,k\} \ | \ i, j, k \in [k+6], \ i \neq j,k \ \text{and} \ j \neq k\}$ and $(v_1,v_2) \in E(KG_{3,k})$ if and only if $v_1 \cap v_2 =\emptyset$.

A vertex $\{i,j,k\}$ of $KG_{3,k}$ is said to be {\it stable} if $ \{t,t+1\} \nsubseteq \{i,j,k\}$ for $t \in [k+6]$, with addition modulo $k+6.$ If $\{i,j,k\}$ is not stable, then it is said to be {\it unstable}. Let $V_{u}^k$ is collection of all unstable vertices of $KG_{3,k}$. In this article, $V_{u}^k$ has the lexicographic ordering of triples.

The induced subgraph of $KG_{3,k}$, with all stable vertices is called the {\it stable Kneser graph} $SG_{3,k}$. These graphs are vertex critical,   {\it i.e.}, the chromatic number of any subgraph of $SG_{3,k}$ obtained by removing a vertex is strictly less than the chromatic number of $SG_{3,k}$ \cite{Sch78}.
\end{Def}

\begin{Def} The subgraph
$S_{3,k}$ of $KG_{3,k}$ is defined to be the graph with $V(S_{3,k})= V(KG_{3,k})$ and
$E(S_{3,k})= \{(u,v) \in E(KG_{3,k})\ | \ \text{at least one of}\ u \ \text{or} \ v\ \text{is stable} \}$.
\end{Def}

In this article, the vertex $\{i,j,k\}$ of $KG_{3,k}$ will be denoted by $ijk$.
For $\sigma =\{\alpha_1, \alpha_2, \ldots \alpha_{t}\} \in \mathcal{N}(KG_{3,k})$, let 
\begin{enumerate}
\item[$(a)$] $\sigma \oplus j = \{ \alpha_{1}  \oplus  j, \  \alpha_{2}  \oplus  j, \ \ldots , \ \alpha_{t}\oplus j\}$ and
\item[$(b)$] $\si \ominus j = \{ \alpha_{1}  \ominus  j, \ \alpha_{2}  \ominus  j,\  \ldots ,\  \alpha_{t} \ominus j \}$. 
\end{enumerate}
If $\alpha_{i} =i_1i_{2}i_3$, then $\alpha_{i} \oplus j =(i_1 +j)(i_{2}+j)(i_3+j)$ and $\alpha_{i} \ominus j = (i_1 -j)(i_{2}-j)(i_3-j)$
with addition  and subtraction being modulo $k+6$.
  For example, in $\mathcal{N}(KG_{3,3})$, $679 \oplus 1 = 781 $, $234 \oplus 1 = 345$, $679 \ominus 1 = 568 $ and $123 \ominus 1 = 912=129$ (see \cite{BZ14}).

If $\sigma = \{\alpha_1, \alpha_2,\ldots,\alpha_t\} \in \mathcal{N}(KG_{3,k})$, then define 
\begin{equation}\label{eq2.1}
S_{\sigma} = \overset{t}{\underset{i=1}{\bigcup}} \ \alpha_i \ \text{and} \ \ C_\sigma = [k+6] \setminus S_{\sigma}.
\end{equation}
It is a simple observation that $C_{\sigma \setminus \{abc\}} \subseteq C_\sigma \cup \ \{a,b,c\}$ and $C_{\sigma \ \cup \ \{abc\}} = C_\sigma \setminus \{a,b,c\} .$

The neighborhood complexes $\mathcal{N}(SG_{3,k})$, $\mathcal{N}(S_{3,k})$ and $\mathcal{N}(KG_{3,k})$ will denote the neighborhood complexes in $SG_{3,k}$, $S_{3,k}$ and $KG_{3,k}$, respectively. Further, addition and subtraction on vertices are $(\text{mod}\ k+6),$ {\it i.e.}, $340 \equiv 34(k+6) (\text{ mod } \ k+6)$ and $35(k+7) \equiv 351 (\text{ mod } \ k+6)$.

We observe the following:

\begin{prop}\label{prop2.1}
\begin{enumerate}
 \item[$(i)$] $\alpha \in V(KG_{3,k}) \setminus V(SG_{3,k}) \rar \exists \ j \in [k+5]\cup \{0\}, \ell \in [k+5]\sm \{1,2\}$, such that $\alpha = 12\ell \oplus j$, {\it i.e.} $ N(\alpha)= \{u \oplus j \ | \ u \in N( 12\ell)\} $.
 
\item[$(ii)$] $\alpha$, $ \beta \in V(KG_{3,k})$ and $\alpha = \beta \oplus j$, $  j \in [k+5]\cup \{0\}\rar N(\alpha)= \{u \oplus j \ | \ u \in N(\beta)\} $.
\end{enumerate}
\end{prop}
\begin{proof}
Firstly, $\alpha = \alpha_1\alpha_2\alpha_3\in V(KG_{3,k}) \setminus V(SG_{3,k})$ implies that there exists $i \in [k+6]$ such that $\{i,i+1\} \subsetneq \{\alpha_1,\alpha_2,\alpha_3\}$. Without loss of generality, assume that $\alpha_1=i$ and $\alpha_2=i+1$. Thus, $\alpha=12(\alpha_3-i+1)\oplus (i-1).$ This proves $(i)$.

Let $\alpha = \beta \oplus j$, $  j \in [k+5]\cup \{0\}$. Here, $u \in N(\beta)$ implies that $\beta \subseteq C_u$. So, $\beta\oplus j \subseteq C_{u \oplus j}$ implying that $u \oplus j \in N(\beta \oplus j) = N(\alpha).$ Further, if $v \in N(\alpha)$ then $v \ominus j \in N(\beta)$. This completes the proof of Proposition \ref{prop2.1}.
\end{proof}
Proposition \ref{prop2.1} holds even when  $KG_{3,k}$ is replaced by $S_{3,k}$.

We now discuss some basic results  that we have used from Discrete Morse theory.

\begin{Def}[\cite{BZ14}]
A partial matching in a poset $(P,\text{ \textless})$ is a subset $M \subseteq P \times P$, where
\begin{enumerate}
\item[(i)] $(a,b) \in M$ implies $ b \succ a;$ {\it i.e.}, $a \ \text{\textless} \ b$ and $ \nexists $ $c \in P$ such that $ a \ \text{\textless}  \ c \ \text{\textless} \ b $ and

\item[(ii)] each $ a \in P $ belongs to at most one element in $M$.
\end{enumerate}

If $(a,b) \in M $, we denote $ a $ by $ d(b) $ and $ b$ by $ u(a) $. A partial matching on $P$ is called acyclic if there do not exist distinct $a_{i}\in P$, $1\leq i \leq  m$, $ m \geq 2 $ such that 
\begin{equation} \label{eq3.1}
 a_1 \ \prec \ u(a_1) \ \succ \ a_2 \ \prec \ u(a_2) \ \succ \ \ldots \ \prec \ u(a_m) \ \succ \ a_1 \  \ {\text is\ \ a\ \ cycle}.
\end{equation}
\end{Def}

For an acyclic matching $M$ on $P$, the unmatched elements of $P$ are called {\it critical}. If each element of $P$ belongs to a member of $M$, then $M$ is called a {\it perfect} matching on $P$. The following results are used to construct acyclic partial matchings.

Let $\Delta \subseteq 2^X$ and $x \in X$. Define 
\begin{equation}\label{eq3.2}
\begin{aligned}
& M(\Delta)_x =\{(\sigma \setminus \{x\} , \ \sigma \cup \ \{x\}) \ | \  \sigma \setminus\{ x\}, \ \sigma \cup \ \{x\} \in \Delta\} ~~and
\\
& (\Delta)_x= \{\sigma \in \Delta \ | \ \sigma \setminus \{x\},\ \sigma \cup  \ \{x\} \in \Delta\}.
\end{aligned}
\end{equation}

\begin{lem}[\cite{na}]\label{pac}
The set $M(\Delta)_x$ is a perfect acyclic matching on $(\Delta)_x$.
\end{lem}

\begin{lem}[\cite{Jon08}]\label{lem33}
Let $M'$ be an acyclic matching on $\Delta'  = \Delta \setminus (\Delta)_x$. Then, $M = M' \cup \ M(\Delta)_x$ is an acyclic matching on $\Delta$.
\end{lem}

\begin{lem}[Cluster Lemma \cite{BB11}]\label{lem34}
If $ \varphi \ : \ P \rightarrow Q$ is an order-preserving map and for each $q \in Q$, the subposet $\varphi^{-1}(q)$ carries an acyclic matching $M_q$, then ${\underset{q \ \in \ Q }{\bigsqcup}} M_q $ is an acyclic matching on P.
\end{lem}

The following result by Forman \cite{For98}, one of the main results of Discrete Morse Theory used in this article, gives a Morse Complex  comprising the critical cells corresponding to the acyclic matching defined on a face poset.

\begin{thm}[Forman \cite{For98}]\label{lem35}
Let $ \Delta $ be a simplicial complex and M be an acyclic matching on the face poset of $\Delta $. Let $c_i$ denote the number of critical $i$-dimensional cells of $\Delta$.
The space $\Delta $ is homotopy equivalent to a cell complex $ \Delta _c$ with $c_i$ cells of dimension $i$ for each $i \geq 0$, plus a single $0$-dimensional cell in the case where the empty set is also paired in the matching.
\end{thm}

This result gives the following corollaries.

\begin{cor}[\cite{BB11}]\label{rem36}
If an acyclic matching has critical cells only in a fixed dimension $i$, then $\Delta$ is homotopy equivalent to a wedge of $i$-dimensional spheres.
\end{cor}

\begin{cor}[\cite{BZ14}]\label{rem37}
If the critical cells of an acyclic matching on $\Delta$ form a subcomplex $\Delta'$ of $\Delta$, then $\Delta $ simplicially collapses to $\Delta'$, implying that $\Delta'$ is a deformation retract of $\Delta$.
\end{cor}

Henceforth, $H_n(X)$ will represent the reduced $n^{th}$ homology group of $X$ with integer coefficients. If $X$ is a cell complex and A is a nonempty subcomplex, we say that $(X,A)$ is a good pair. We now present some results for a good pair $(X,A)$ from \cite{hatcher}.

\begin{prop}[Proposition $2.22$, page $124$, \cite{hatcher}]\label{prop3.1}
The quotient map $q: (X,A) \longrightarrow (X/A,A/A)$ induces isomorphisms $q_{*}:H_n(X,A)\longrightarrow H_n(X/A,A/A) \cong H_n(X/A).$
\end{prop} 
 The good pair $(X,A)$ also gives a long exact sequence.
 
\begin{thm}[Theorem $2.13$, page 114, \cite{hatcher}]\label{thm3.4}
There is an exact sequence
\begin{equation*}
\ldots \rightarrow H_n(A) \rightarrow H_n(X) \rightarrow H_n(X,A)\rightarrow H_{n-1}(A) \rightarrow \ldots \rightarrow H_0(X,A) \rightarrow 0.
\end{equation*}
\end{thm}
 
\begin{cor}[Corollary $2.14$, page $114$, \cite{hatcher}]\label{cor3.5}
 $H_n(S^n)\cong \mathbb{Z}$ and $H_i(S^n)=0$ for $i\neq n,$ where $S^n$ is the $n$-dimensional sphere.
\end{cor}

The short exact sequence $0\rightarrow A \rightarrow B \rightarrow C \rightarrow 0$ is said to {\it split}, if $B \cong A\bigoplus C$. 

\begin{lem}[Page $148$, \cite{hatcher}]\label{lem3.4}
If $C$ is free then every exact sequence $0\rightarrow A \rightarrow B \rightarrow C \rightarrow 0$ splits.
\end{lem}

In this article, $\si \sm\{ ijk\}$ and $\si \cup\ \{ ijk\}$ will indicate that $ijk \in \si$ and $ijk \nin \si$, respectively. For brevity, we use $\si \sm ijk$ and $\si \cup\  ijk$  instead of $\si \sm\{ ijk\}$ and $\si \cup\ \{ ijk\}$, respectively. Further, $\{a_1 <a_2<a_3    < \ldots < a_{k}\}$ will denote a chain.


\section{Proof of Theorem \ref{thm2}}
In this section, ${N}(\si)$ denotes the neighborhood of $\si$ in $ S_{3,k}$. 
For $s \in [k+6],$ define
\begin{equation}\label{eq41}
\begin{split}
I_{s} &=[k+6] \sm \{s-1, s, s+1\} \text{ and}\\ 
J_{s} & = \left\{\def\arraystretch{1.2}%
  \begin{array}{@{}c@{\quad}l@{}}
   [k+5]\sm [2] \hspace{0.5cm} & \text{if} ~~s=1,  \\
    ~[k+6] \sm [s+1] & \text{if} ~~1<s<k+5,\\
    \emptyset & \text{otherwise}.\\   
  \end{array}\right.
  \end{split}
\end{equation}
 
Let $V_{u}^k$ denote the set of all the unstable vertices of $KG_{3,k}$. For $t \in I_s$ and $u \in J_s$, define
\begin{equation}\label{eq42}
\begin{aligned}
 A_k^{s,t} & =\{\si \in N(S_{3,k}) \ | \ C_\si =\{s,s+1,t\}\},
\\
 B_k^{s,u} & =\{\si \in N(S_{3,k}) \ | \ C_\si =\{s,s+1,u,u+1\}\} ~~and
\\
C_k &=\{\si \in N(S_{3,k}) \ | \ \si\cap V_{u}^k \neq \emptyset\}.
\end{aligned}
\end{equation}

From the above defintions, the following is an easy consequence.
\begin{prop}\label{prop41}
Let $s_1,s_2 \in [k+6]$, $t_1 \in I_{s_1}, t_2 \in I_{s_2}, u_1 \in J_{s_1}$ and $u_2 \in J_{s_2}$. Then,
\begin{enumerate}

   \item[$(i)$] $A_k^{s_1,t_1} \cap A_k^{s_2,t_2}\neq \emptyset \Longleftrightarrow s_1=s_2$ and $t_1=t_2$.
   \item[$(ii)$] $B_k^{s_1,u_1} \cap B_k^{s_2,u_2}\neq \emptyset \Longleftrightarrow s_1=s_2$ and $u_1=u_2$.
   \item[$(iii)$] $A_k^{s_1,t_1} \cap B_k^{s_2,u_2}= \emptyset$.
  \end{enumerate}
\end{prop}

Using the above proposition, we construct a disjoint decomposition of subsets of $\N(S_{3,k})$.

\begin{prop}\label{prop42}
\begin{enumerate}
\item[$(i)$] $\si \in \N(S_{3,k})$ and $ N(\si)\cap V_{u}^k\neq \emptyset \rar \si \subseteq V(SG_{3,k})$. 

\item[$(ii)$] $\N(S_{3,k})= \N(SG_{3,k})\bigsqcup( \bigsqcup\limits_{s \in [k+6]}\bigsqcup\limits_{t \in I_s}A_{k}^{s,t})\bigsqcup( \bigsqcup\limits_{s \in [k+4]}\bigsqcup\limits_{u \in J_s}B_{k}^{s,u})\bigsqcup C_k$.

\item[$(iii)$] If $\varphi_k : \mathcal{F}(\N(S_{3,k})) \longrightarrow \{a_k^1 > a_k^2 >a_k^3 >a_k^4\}$ defined as ,

\begin{equation}\label{eqn43}
  \varphi_k(\sigma) = \left\{\def\arraystretch{1.2}%
  \begin{array}{@{}c@{\quad}l@{}}
    a_k^1 & \text{if} ~~\si \in \bigsqcup\limits_{s \in [k+6]}\bigsqcup\limits_{t \in I_s}A_{k}^{s,t} ,  \\
    a_k^2 & \text{if} ~~\si \in \bigsqcup\limits_{s \in [k+4]}\bigsqcup\limits_{u \in J_s}B_{k}^{s,t},\\
    a_k^3 & \text{if} ~~\si \in C_k,\\
    a_k^4 & \text{if} ~~\si \in \mathcal{F}(\N(SG_{3,k})),\\
  \end{array}\right.
\end{equation}
then $\varphi_k$ is a poset map.
\end{enumerate}
\end{prop}
\begin{proof}
No two unstable vertices are adjacent in $S_{3,k}$. Thus, $\si \subseteq V(SG_{3,k})$, thereby proving $(i)$. 

Let $\si \in \N(S_{3,k})$. If $N(\si)\cap V(SG_{3,k})\neq \emptyset$ and $\si \cap V_{u}^k=\emptyset$, then from Proposition \ref{prop42}$(i)$, $\si \in \N(SG_{3,k})$. If $\si \cap V_{u}^k\neq \emptyset$, then $\si \in C_k$. Now let, $N(\si)\cap V(SG_{3,k})=\emptyset$, {\it i.e.}, $N(\si)\subset V_{u}^k$. Here, $C_\si$ is either $\{s,s+1,t\}$ or $\{s,s+1,u,u+1\}$ (otherwise $C_\si$ will contain a stable $3$ set). Therefore, $\si$ is in either $ A_{k}^{s,t}$ or $B_{k}^{s,u}$. The result now follows by using Equation \eqref{eq42}.

Consider, $\tau, \si \in  \mathcal{F}(\N(S_{3,k}))$ with $\tau \subseteq \si.$ Since $C_\si \subseteq C_\tau$, we see that if $C_\tau = \{s,s+1,t\}$, then $C_\si = \{s,s+1,t\}$. Further, if $C_\tau=\{s,s+1,u,u+1\}$, then $N(\si)\subseteq V_{u}^k$, implying that $\varphi_k(\si) \geq \varphi_k(\tau)$. Let, $\tau \in \varphi_k^{-1}(a_k^3)$, {\it i.e.}, $\tau \in C_{k}$. Here, $\tau \cap V_{u}^k\neq \emptyset$ implies that $\si \cap V_{u}^k\neq \emptyset$ (since $\tau \subseteq \si$). Therefore, $\si \in C_k$, implying that $\varphi_k$ is a poset map.
\end{proof}

We will now construct a perfect acyclic matching on $\varphi_k^{-1}(a_k^i)$, for $i=1,2,3$.   First, consider the case $i=1$. 
If $\si \in A_{k}^{s_1,t_1}, \tau \in A_{k}^{s_2,t_2}$ and  $\tau \subseteq \si$, then from Proposition \ref{prop41}$(i)$, $s_1=s_2$ and $t_1=t_2.$ Therefore, to define a perfect acyclic matching on $\varphi_k^{-1}(a_k^1)$, it is sufficient to define a perfect acyclic matching on $A_{k}^{s,t}$ for $s \in [k+6]$ and $t \in I_s$. We first construct the matchings on $A_k^{1,\ell}$, $\ell \in \{3,4,\ldots,k+5\}, \ k\geq 1.$

\begin{prop}\label{prop21} There exist perfect acyclic matchings, $M_k^{1,\ell}$ on $A_{k}^{1,\ell}$, $\ell \in \{3,4, \ldots , k+5\}$, $k\geq 1$.
\end{prop}

\begin{proof} The matching is constructed by the method of induction on $k$.
 
 Let $k=1$. 
Here, $N(12\ell) = \emptyset$ for $\ell \in \{3,5\}$ and $N(124)=N(126)=\{357\}.$ In each of these cases, $C_{N(12\ell)}\neq \{1,2,\ell\} $. Thus, $A_1^{1,\ell}=\emptyset $ for $\ell \in \{3,4,5,6\}$.

If $k=2$, then
\[
  A_2^{1,i}= \left\{\def\arraystretch{1.2}%
  \begin{array}{@{}c@{\quad}l@{}}
    \{\{357,368\},\{357,358,368\}\} & \text{if } i=4,\\
    \{\{358,468\},\{358,368,468\}\} & \text{if }i=7,\\
    \emptyset & \text{if } i \in \{3,5,6\}.\\
  \end{array}\right.
\]
Hence, $(A_2^{1,4})_{358} =A_2^{1,4}$ and $(A_2^{1,7})_{368} =A_2^{1,7}$ {\it i.e.}, $M(A_2^{1,4})_{358}$ and $M(A_2^{1,7})_{368}$ are perfect acyclic matchings on $A_2^{1,4}$ and $A_2^{1,7}$, respectively.

Inductively, assume that for $r \geq 3$, $ k \in [r-1]  $ and $\ell \in \{3,\ldots, k+5\}$, there exist perfect acyclic matchings on $A_k^{1,\ell}$. By Proposition \ref{prop2.1}$(i)$ and $(ii)$, these matchings induce perfect acyclic matchings on $A_k^{s,t}$, for each $k \in [r-1], s \in [k+6] $ and $t \in I_s$. 

Now consider $k=r.$  Here, for each $\ell \in \{3, \ldots, r+5\}$, we define an element matching $(A_r^{1,\ell})_v$ on $A_r^{1,\ell}$, where the vertex $v \in V(SG_{3,r})$ is chosen in such a way that $A_r^{1,\ell} \sm (A_r^{1,\ell})_{v}$ can be partitioned into smaller sets. On these subsets, using the existing matchings obtained above, we construct the required matchings. 

\begin{enumerate} 
\item Let $\ell=3$. Firstly, $M(A_r^{1,3})_{46(r+6)}$ is an acyclic matching on $A_r^{1,3}$.  We now define a matching on the complement $ A_r^{1,3} \sm(A_r^{1,3})_{46(r+6)}.$

\begin{claim}\label{claim1}
$\sigma \in A_r^{1,3} \sm (A_r^{1,3})_{46(r+6)}\rar 46(r+6) \in \si$ and $C_{\si \sm 46(r+6)} \neq \{1,2,3\}$.
\end{claim}
Suppose that $46(r+6) \nin \si.$ Here, $C_\si=\{1,2,3\}$ implies that $C_{\si \cup 46(r+6)}=\{1,2,3\},$ thereby showing that $\si \in (A_r^{1,3})_{46(r+6)},$ a contradiction. Thus, $46(r+6) \in \si.$ 

Further, $C_{\si \sm 46(r+6)}\neq \{1,2,3\}$ (as $C_{\si\sm 46(r+6)}=\{1,2,3\} \rar \si, \si\sm 46(r+6) \in A_r^{1,3} \rar \si \in (A_r^{1,3})_{46(r+6)}$). This completes the proof of Claim \ref{claim1}.

From Claim \ref{claim1}, $C_{\si \sm 46(r+6)} \bigcap \{4,6,r+6\} \neq \emptyset$. Thus, we have a disjoint decomposition of $A_r^{1,3} $ as $A_r^{1,3}= (A_r^{1,3})_{46(r+6)} \bigsqcup (\bigsqcup\limits_{i=1 }^{7} \Delta^3_i)$, where
\begin{equation}\label{gamma123}
\begin{split}
\Delta^3_1 & =\{\sigma \in A_r^{1,3} \ | \ C_{\sigma \setminus 46(r+6)}= \{1,2,3,4\}\},\\
\Delta^3_2 & =\{\sigma \in A_r^{1,3} \ | \ C_{\sigma \setminus 46(r+6)}=\{1,2,3,6\}\}, \\
\Delta^3_3 & =\{\sigma \in A_r^{1,3} \ | \ C_{\sigma \setminus 46(r+6)}=\{1,2,3,r+6\}\}, \\
\Delta^3_4 & =\{\sigma \in A_r^{1,3} \ | \ C_{\sigma \setminus 46(r+6)}=\{1,2,3,4,6\}\}, \\
\Delta^3_5 & =\{\sigma \in A_r^{1,3} \ | \ C_{\sigma \setminus 46(r+6)}=\{1,2,3,4,r+6\}\}, \\ 
\Delta^3_6 & =\{\sigma \in A_r^{1,3} \ | \ C_{\sigma \setminus 46(r+6)}=\{1,2,3,6,r+6\}\} \text{ and } \\
\Delta^3_7 & =\{\sigma \in A_r^{1,3} \ | \ C_{\sigma \setminus 46(r+6)}=\{1,2,3,4,6,r+6\}\}.
\end{split}
\end{equation}
The following gives a bijective correspondence from each of the above ${\Delta_i^3}^{'}$s to a  $A_\beta^{1,j}$, where $\beta <r$ and $j \in\{3,4,5\}$. Here $A_{\beta}^{1,j}$ is considered as a subset of $\N (S_{3,r})$, {\it i.e.},
\begin{equation}\label{eq2.7}
\tau \in A_\beta^{1,j} \rar C_\tau = \{1,2,j,\beta+7,\beta+8,\ldots, r+6\}.
\end{equation}

\begin{claim}\label{cla4.2}
The following maps are bijective:
\begin{enumerate}
\item $f_1 :  \Delta^3_{1} \rightarrow A_{r-1}^{1,3} $ defined by $f_1(\si)=\{\si \sm 46(r+6)\}\ominus 1$, $r >2$.
\item $f_2 : \Delta^3_{2} \rightarrow A_{r-1}^{1,5}$ defined by $f_2(\si)=\{\si \sm 46(r+6)\}\ominus 1$, $r >2$.
\item $f_3 : \Delta^3_{3} \rightarrow A_{r-1}^{1,3}$ defined by $f_3(\si)=\{\si \sm 46(r+6)\}$, $r >2$.
\item $f_4 : \Delta^3_{4} \rightarrow A_{r-2}^{1,4}$ defined by $f_4(\si)=\{\si \sm 46(r+6)\}\ominus 2 $, $r >2$.
\item $f_5 : \Delta^3_{5} \rightarrow A_{r-2}^{1,3}$ defined by $f_5(\si)=\{\si \sm 46(r+6)\}\ominus 1 $, $r >2$.
\item $f_6 : \Delta^3_{6} \rightarrow A_{r-2}^{1,5}$ defined by $f_6(\si)=\{\si \sm 46(r+6)\}\ominus 1$, $r >2$.
\item $f_7 : \Delta^3_{7} \rightarrow A_{r-3}^{1,4}$ defined by $f_7(\si)=\{\si \sm 46(r+6)\}\ominus 2$, $r >3$.
\end{enumerate}
\end{claim}

\begin{enumerate}
\item[$(1)$] $f_i$ is well defined for all $i \in \{1,2,\ldots,7\}.$ 

Let $\si_i\in \Delta^3_i$ and $f_i(\si_i)=\tau_i$, for each $i \in [7]$. 
In $\N(S_{3,r})$, $C_{\si_1 \sm 46(r+6)} =\{1,2,3,4\},$ $C_{\si_2 \sm 46(r+6)} =\{1,2,3,6\},$ $C_{\si_3 \sm 46(r+6)} =\{1,2,3,r+6\}$, $C_{\tau_1} =\{r+6,1,2,3\}, \ C_{\tau_2} =\{r+6,1,2,5\}$ and $C_{\tau_3} =\{1,2,3,r+6\}$. Thus, in $\N(S_{3,r-1})$, $C_{\tau_1} =C_{\tau_3}=\{1,2,3\}$ and $C_{\tau_2}=\{1,2,5\}$. Therefore, by Equation \eqref{eq42}, $\tau_1, \tau_3 \in A_{r-1}^{1,3}$ and $\tau_2 \in A_{r-1}^{1,5}$ implying that $f_1, f_2$ and $f_3$ are well defined maps. 
Similarly, it can be shown that$f_4, f_5, f_6$ and $f_7$ are also well defined.

\item[$(2)$] $f_i$ is injective for all $i \in \{1,2,\ldots,7\}.$ 

Let $\si_1, \si_2 \in \Delta^3_i$ such that  $f_i(\sigma_1)=f_i(\si_2)$. From the definition of the $f_i {'s}$, there exists $t\in \{0,1,2\}$ such that $\{\si_1 \sm 46(r+6)\}\ominus t=\{\si_2 \sm 46(r+6)\} \ominus t$, {\it i.e.}, $\si_1\sm 46(r+6)=\si_2\sm 46(r+6)$. Thus, $\si_1=\si_2$.

\item[$(3)$] $f_i$ is surjective for all $i \in \{1,2,\ldots,7\}.$ 

Let $\tau \in A_{r-1}^{1,3}$. In $\N(S_{3,r-1})$, $C_\tau=\{1,2,3\}$. From Equation \eqref{eq2.7}, $C_\tau=\{1,2,3,r+6\}$ in $\N(S_{3,r})$, {\it i.e.} $46(r+6) \nin \tau$, and $C_{\tau \cup 46(r+6)}= \{1,2,3\}$ {\it i.e.}, $\tau \cup 46(r+6) \in \Delta^3_3$. Further, $C_{\tau\oplus 1}=\{1,2,3,4\}$ and $C_{\{\tau\oplus 1\} \cup \{46(r+6)\}}=\{1,2,3\},$ {\it i.e.}, $\{\tau\oplus 1 \} \cup \{46(r+6)\} \in \Delta^3_1$. Therefore, $f_1(\{\tau\oplus 1 \} \cup \{46(r+6)\})=f_3(\tau \cup 46(r+6))=\tau$ implying that $f_1$ and $f_3$ are surjective maps. By similar arguments, all the other maps are also surjective.
\end{enumerate}
To construct the perfect acyclic matching on $A_r^{1,3}$, 
define,\newline $\theta_3 : A_r^{1,3} \longrightarrow \{ q_7 < q_6 < q_5 <q_4 <q_3 <q_2 <q_1 < q_0 \}$ by
\begin{equation}\label{eq4.5}
  \theta_{3}(\sigma) = \left\{\def\arraystretch{1.2}%
  \begin{array}{@{}c@{\quad}l@{}}
   q_i & \text{if $\sigma \in \Delta_i^3, \ 1\leq i \leq 7$, }\\
   q_{0}& \text{otherwise.}
  \end{array}\right.
\end{equation}
Firstly, we see that if $\si \in \Delta_{i_1}^3$ and $\tau \in \Delta_{i_2}^3$ with $\tau \subseteq \si$, then $i_1\leq i_2$.

Let $\si, \tau \in A_r^{1,3}, \tau \subseteq \si$ and $\tau \in {\theta_3}^{-1}(q_0)$. If $\si \in A_r^{1,3} \sm (A_r^{1,3})_{46(r+6)},$ then from Claim \ref{claim1}, there exists $i \in [7]$, such that $\si \in \Delta_i^3$. Therefore, $\gamma \in {\theta_3}^{-1}(q_0)$ if and only if $C_{\gamma \sm 46(r+6)} = \{1,2,3\}$. Since $C_{\tau \sm 46(r+6)}=\{1,2,3\}$, $\forall \ \tau \in {\theta_3}^{-1}(q_0)$, we observe that $\tau \subseteq \si$ implies that $C_{\si \sm 46(r+6)}=\{1,2,3\}$, {\it i.e.}, $\si \in {\theta_3}^{-1}(q_0)$. Therefore,  $\theta_3$ is a poset map.

By the induction hypothesis, there exist perfect acyclic matchings $M(A_{\beta}^{s,t})$ on $A_\beta^{s,t}$ for $\beta<r$ and $s,t \in [\beta+6]$. Claim \ref{cla4.2} induces the perfect acyclic matchings $f_{i}^{-1} (M(A_{\beta}^{s,t}))= M_i^3$ on $\Delta_i^3$, $\forall \ i \in [7]$. Since $\theta_{3}$ is a poset map, from Lemma \ref{lem34}, $M_r^{1,3}=(\bigsqcup \limits_{i=1}^{7} M_i^3) \bigsqcup M(A_r^{1,3})_{46(r+6)}$ is a perfect acyclic matching on $A_r^{1,3}$.

\item $\ell=4$. As $M(A_r^{1,4})_{35(r+6)}$ is an acyclic matching on $A_r^{1,4}$, we consider the complement $ A_r^{1,4} \sm(A_r^{1,4})_{35(r+6)}.$ From an argument similar to that in Claim \ref{claim1}, we see that
$\sigma \in A_r^{1,4} \sm (A_r^{1,4})_{35(r+6)}\rar 35(r+6) \in \si$ and $C_{\si \sm 35(r+6)} \neq \{1,2,4\}$. This implies that $C_{\si \sm 35(r+6)} \bigcap \{3,5, \\ r+6\} \neq \emptyset$. Thus, we have a disjoint decomposition of $A_r^{1,4} $ as $A_r^{1,4}= (A_r^{1,4})_{35(r+6)} \bigsqcup (\bigsqcup\limits_{i=1 }^{7} \Delta^4_i)$, where
$\Delta^4_1 =\{\sigma \in A_r^{1,4} \ | \ C_{\sigma \setminus 35(r+6)}= \{1,2,3,4\}\},$
$\Delta^4_2 =\{\sigma \in A_r^{1,4} \ | \ C_{\sigma \setminus 35(r+6)}=\{1,2,4,5\}\},$
$\Delta^4_3 =\{\sigma \in A_r^{1,4} \ | \ C_{\sigma \setminus 35(r+6)}=\{1,2,4,r+6\}\},$
$\Delta^4_4 =\{\sigma \in A_r^{1,4} \ | \ C_{\sigma \setminus 35(r+6)}=\{1,2,3,4,5\}\},$
$\Delta^4_5 =\{\sigma \in A_r^{1,4} \ | \ C_{\sigma \setminus 35(r+6)}=\{1,2,3,4,r+6\}\},$
$\Delta^4_6 =\{\sigma \in A_r^{1,4} \ | \ C_{\sigma \setminus 35(r+6)}=\{1,2,4,5,r+6\}\}$  and
$\Delta^4_7 =\{\sigma \in A_r^{1,4} \ | \ C_{\sigma \setminus 35(r+6)}=\{1,2,3,4,5,r+6\}\}.$

As in the case, $\ell =3$, it can be shown that the following maps are bijective:
\begin{enumerate}
\item $f_1 :  \Delta^4_{1} \rightarrow A_{r-1}^{1,3} $ defined by $f_1(\si)=\{\si \sm 35(r+6)\}\ominus 1$, $r >2$.
\item $f_2 : \Delta^4_{2} \rightarrow A_{r-1}^{3,1}$ defined by $f_2(\si)=\{\si \sm 35(r+6)\}\ominus 1$, $r >2$.
\item $f_3 : \Delta^4_{3} \rightarrow A_{r-1}^{1,4}$ defined by $f_3(\si)=\{\si \sm 35(r+6)\}$, $r >2$.
\item $f_4 : \Delta^4_{4} \rightarrow A_{r-2}^{1,3}$ defined by $f_4(\si)=\{\si \sm 35(r+6)\}\ominus 2 $, $r >2$.
\item $f_5 : \Delta^4_{5} \rightarrow A_{r-2}^{1,3}$ defined by $f_5(\si)=\{\si \sm 35(r+6)\}\ominus 1 $, $r >2$.
\item $f_6 : \Delta^4_{6} \rightarrow A_{r-2}^{3,1}$ defined by $f_6(\si)=\{\si \sm 35(r+6)\}\ominus 1$, $r >2$.
\item $f_7 : \Delta^4_{7} \rightarrow A_{r-3}^{1,3}$ defined by $f_7(\si)=\{\si \sm 35(r+6)\}\ominus 2$, $r >3$.
\end{enumerate}

By arguments similar to those in the case $\ell=3$, $M_r^{1,4}=(\bigsqcup \limits_{i=1}^{7} M_i^4) \bigsqcup M(A_r^{1,4})_{35(r+6)}$ is a perfect acyclic matching on $A_r^{1,4}$.

\item $\ell=5$. In this case, $M(A_r^{1,5})_{36(r+6)}$ is an acyclic matching on $A_r^{1,5}$.  As in the earlier cases
$\sigma \in A_r^{1,5} \sm (A_r^{1,5})_{36(r+6)} \rar 36(r+6) \in \si$ and $C_{\si \sm 36(r+6)} \neq \{1,2,5\}$. Further,
 $A_r^{1,5}= (A_r^{1,5})_{36(r+6)} \bigsqcup (\bigsqcup\limits_{i=1 }^{7} \Delta^5_i)$, where
$\Delta^5_1  =\{\sigma \in A_r^{1,5} \ | \ C_{\sigma \setminus 36(r+6)}= \{1,2,3,5\}\},$
$\Delta^5_2 =\{\sigma \in A_r^{1,5} \ | \ C_{\sigma \setminus 36(r+6)}=\{1,2,5,6\}\}, $
$\Delta^5_3  =\{\sigma \in A_r^{1,5} \ | \ C_{\sigma \setminus 36(r+6)}=\{1,2,5,r+6\}\},$
$\Delta^5_4 =\{\sigma \in A_r^{1,5} \ | \ C_{\sigma \setminus 36(r+6)}=\{1,2,3,5,6\}\},$
$\Delta^5_5 =\{\sigma \in A_r^{1,5} \ | \ C_{\sigma \setminus 36(r+6)}=\{1,2,3,5,r+6\}\},$
$\Delta^5_6 =\{\sigma \in A_r^{1,5} \ | \ C_{\sigma \setminus 36(r+6)}=\{1,2,5,6,r+6\}\}$ and
$\Delta^5_7  =\{\sigma \in A_r^{1,5} \ | \ C_{\sigma \setminus 36(r+6)}=\{1,2,3,5,6,r+6\}\}.$
The following maps are bijective:
\begin{enumerate}
\item $f_1 :  \Delta^5_{1} \rightarrow A_{r-1}^{1,4} $ defined by $f_1(\si)=\{\si \sm 36(r+6)\}\ominus 1$, $r >2$.
\item $f_2 : \Delta^5_{2} \rightarrow A_{r-1}^{4,1}$ defined by $f_2(\si)=\{\si \sm 36(r+6)\}\ominus 1$, $r >2$.
\item $f_3 : \Delta^5_{3} \rightarrow A_{r-1}^{1,5}$ defined by $f_3(\si)=\{\si \sm 36(r+6)\}$, $r >2$.
\item $f_4 : \Delta^5_{4} \rightarrow A_{r-2}^{3,1}$ defined by $f_4(\si)=\{\si \sm 36(r+6)\}\ominus 2 $, $r >2$.
\item $f_5 : \Delta^5_{5} \rightarrow A_{r-2}^{1,4}$ defined by $f_5(\si)=\{\si \sm 36(r+6)\}\ominus 1 $, $r >2$.
\item $f_6 : \Delta^5_{6} \rightarrow A_{r-2}^{4,1}$ defined by $f_6(\si)=\{\si \sm 36(r+6)\}\ominus 1$, $r >2$.
\item $f_7 : \Delta^5_{7} \rightarrow A_{r-3}^{3,1}$ defined by $f_7(\si)=\{\si \sm 36(r+6)\}\ominus 2$, $r >3$.
\end{enumerate}
Here, $M_r^{1,5}=(\bigsqcup \limits_{i=1}^{7} M_i^5) \bigsqcup M(A_r^{1,5})_{36(r+6)}$ is a perfect acyclic matching on $A_r^{1,5}$, where $M_i^5$ is the perfect acyclic matching on $\Delta_i^5$ for $i \in [7]$, induced by the above bijective maps. 

\item $5< \ell \leq r+5$. As in the case of $l \in \{3,4,5\}$, we see that $M(A_r^{1,\ell})_{3(\ell-1)(r+6)}$ is an acyclic matching on $A_r^{1,\ell}$ and
$\sigma \in A_r^{1,\ell} \sm (A_r^{1,\ell})_{3(\ell-1)(r+6)} \rar 3(\ell-1)(r+6) \in \si$ and $C_{\si \sm 3(\ell-1)(r+6)} \neq \{1,2,\ell\}$.
This gives $A_r^{1,\ell}= (A_r^{1,\ell})_{3(\ell-1)(r+6)} \bigsqcup (\bigsqcup\limits_{i=1 }^{7} \Delta^{\ell}_i)$, where $\Delta^{\ell}_{i} \subset A_r^{1,\ell}$, $\forall i \in [7]$ and
$ \si \in \Delta^{\ell}_1 \Leftrightarrow C_{\sigma \setminus 3(\ell-1)(r+6)}= \{1,2,3,\ell\}\},$
$\si \in \Delta^{\ell}_2  \Leftrightarrow \ C_{\sigma \setminus 3(\ell-1)(r+6)}=\{1,2,\ell-1,\ell\}\},$
$\si \in \Delta^{\ell}_3  \Leftrightarrow C_{\sigma \setminus 3(\ell-1)(r+6)}=\{1,2,\ell,r+6\}\},$
$ \si \in \Delta^{\ell}_4  \Leftrightarrow C_{\sigma \setminus 3(\ell-1)(r+6)}=\{1,2,3,\ell-1,\ell\}\},$
$ \si \in \Delta^{\ell}_5  \Leftrightarrow C_{\sigma \setminus 3(\ell-1)(r+6)}=\{1,2,3,\ell,r+6\}\},$
$\si \in \Delta^{\ell}_6  \Leftrightarrow C_{\sigma \setminus 3(\ell-1)(r+6)}=\{1,2,\ell-1,\ell,r+6\}\}$  and
$\si \in \Delta^{\ell}_7   \Leftrightarrow C_{\sigma \setminus 3(\ell-1)(r+6)}=\{1,2,3,\ell-1,\ell,r+6\}\}.$
The following maps are bijective:
\begin{enumerate}
\item $f_1 :  \Delta^{\ell}_{1} \rightarrow A_{r-1}^{1,\ell-1} $ defined by $f_1(\si)=\{\si \sm 3(\ell-1)(r+6)\}\ominus 1$, $r >2$.
\item $f_2 : \Delta^{\ell}_{2} \rightarrow A_{r-1}^{\ell-2,1}$ defined by $f_2(\si)=\{\si \sm 3(\ell-1)(r+6)\}\ominus 1$, $r >2$.
\item $f_3 : \Delta^{\ell}_{3} \rightarrow A_{r-1}^{1,\ell}$ defined by $f_3(\si)=\{\si \sm 3(\ell-1)(r+6)\}$, $r >2$.
\item $f_4 : \Delta^{\ell}_{4} \rightarrow A_{r-2}^{\ell-3,1}$ defined by $f_4(\si)=\{\si \sm 3(\ell-1)(r+6)\}\ominus 2 $, $r >2$.
\item $f_5 : \Delta^{\ell}_{5} \rightarrow A_{r-2}^{1,\ell-1}$ defined by $f_5(\si)=\{\si \sm 3(\ell-1)(r+6)\}\ominus 1 $, $r >2$.
\item $f_6 : \Delta^{\ell}_{6} \rightarrow A_{r-2}^{\ell-2,1}$ defined by $f_6(\si)=\{\si \sm 3(\ell-1)(r+6)\}\ominus 1$, $r >2$.
\item $f_7 : \Delta^{\ell}_{7} \rightarrow A_{r-3}^{\ell-3,1}$ defined by $f_7(\si)=\{\si \sm 3(\ell-1)(r+6)\}\ominus 2$, $r >3$.
\end{enumerate}
In this case, $M_r^{1,\ell}=(\bigsqcup \limits_{i=1}^{7} M_i^\ell) \bigsqcup M(A_r^{1,\ell})_{3(\ell-1)(r+6)}$ is a perfect acyclic matching on $A_r^{1,\ell}$ for $6 \leq \ell \leq r+5$, where $M_i^\ell$ is the perfect acyclic matching on $\Delta_i^\ell$ for $i \in [7]$. 
This completes the proof of Proposition \ref{prop21}.
\end{enumerate}
\end{proof}
From Proposition \ref{prop2.1}$(i)$ and $(ii)$, there exists $\ell \in \{3,4,\ldots,k+5\}$, such that $A_{k}^{s,t} = A_k^{1,\ell} \oplus (s-1)$. Therefore, using Proposition \ref{prop21}, $M_k^{s,t}=M_k^{1,\ell}\oplus(s-1)$ is a perfect acyclic matching on $A_k^{s,t}$.  Therefore, we have a perfect acyclic matching on $A_k^{s,t}$ for $s\in [k+6], \ t \in I_s$. This induces the matching on $\varphi_k^{-1}(a_k^1).$ Using Proposition \ref{prop41}$(i)$ and Lemma \ref{lem34}, we have proved that :
  
\begin{lem}\label{lemma21}
$\bigsqcup\limits_{s \in [k+6]}\bigsqcup\limits_{t \in I_s}M_{k}^{s,t}$ is a perfect acyclic matching on $\varphi_k^{-1}(a_k^1)=\bigsqcup\limits_{s \in [k+6]}\bigsqcup\limits_{t \in I_s}A_{k}^{s,t}$. 
\end{lem}
Now, consider $\varphi_k^{-1}(a_k^2) =\bigsqcup\limits_{s \in [k+4]}\bigsqcup\limits_{u \in J_s}B_{k}^{s,u}$, where $B_{k}^{s,u}$ is as defined in Equation \eqref{eq42}.

\begin{lem}\label{lem4.11}  The map $f_{s,u} :  B_{k}^{s,u} \longrightarrow A_{k-1}^{s+k+5-u,k+5},$ $s \in [k+4]$ and $u \in J_s$, defined by $f_{s,u}(\si)=\si \oplus (k+5-u)$ is bijective.  
\end{lem}
\begin{proof}
Since $B_{1}^{s,u}=\emptyset$ for $s \in [5]$ and $u \in J_s$, we assume that $k>1$.
Let $\si \in B_{k}^{s,u}$ and $\tau=f_{s,u}(\si)$. In $\N(S_{3,k})$, from Equation \eqref{eq42}, $C_{\si} =\{s,s+1,u,u+1\}$ and $C_{\tau} =\{s+k+5-u,s+k+6-u,k+5,k+6\}$. Thus, in $\N(S_{3,k-1})$, $C_{\tau}=\{s+k+5-u,s+k+6-u,k+5\}$. Therefore, $\tau \in A_{k-1}^{s+k+5-u,k+5}$, implying that $f_{s,u}$ is a well defined map.  

Let $\si_1, \si_2 \in B_{k}^{s,u}$ and  $f_{s,u}(\sigma_1)=f_{s,u}(\si_2)$. Here, $\si_1\oplus (k+5-u)=\si_2\oplus (k+5-u)$, {\it i.e.}, $\si_1=\si_2$., Thus $f_{s,u}$ is an injective map.

If $\tau \in A_{k-1}^{s+k+5-u,k+5}$, then $C_\tau=\{s+k+5-u,s+k+6-u,k+5\}$ and $\{s+k+5-u,s+k+6-u,k+5,k+6\}$ in $\N(S_{3,k-1})$ and $\N(S_{3,k})$ respectively. This implies that $C_{\tau \ominus (k+5-u)}=\{s,s+1,u,u+1\}$. From Equation \eqref{eq42}, $\tau \ominus (k+5-u) \in B_{k}^{s,u}$. Therefore, $f_{s,u}(\tau \ominus (k+5-u))=\tau$ implying that $f_{s,u}$ is a surjective map.
\end{proof}

The perfect acyclic matching on $A_{k-1}^{s+k+5-u,k+5}$,  $s \in [k+4]$, $u \in J_s$, obtained from Lemma \ref{lemma21}, induces a perfect acyclic matching, $N_{k}^{s,u}$ on $B_{k}^{s,u}$ via the bijective map $f_{s,u}$. 
Using Proposition \ref{prop41}$(ii)$ and Lemma \ref{lem34}, we have the following:

\begin{lem} \label{lem4.2}
 $\bigsqcup\limits_{s \in [k+4]}\bigsqcup\limits_{u \in J_s}N_{k}^{s,u}$ is a perfect acyclic matching on $\varphi_k^{-1}(a_k^2)$.
\end{lem}

Now, consider $\varphi_k^{-1}(a_k^3)=C_k$ where $C_{k}$ is as defined in  Equation \eqref{eq42}.  
For $v \in V_{u}^k$,  where $V_{u}^k$, the set of unstable vertices of $KG_{3,k}$, has the lexicographic ordering of triples, define
\begin{equation}\label{eqU}
C_{k}^{v} =\{\si \in C_k \ | \ v \in \si \text{ and } v^\prime \nin \si, \ \forall \ v^\prime \in V_{u}^k \text{ with } v^\prime <v\}.
\end{equation}

\begin{prop}\label{prop43}
If $V_{u}^k=\{v_1<v_2<\ldots <v_m\}$, then
\begin{enumerate}
\item[$(i)$] $C_k=\bigsqcup\limits_{v \in V_{u}^k} C_{k}^v $.
\item[$(ii)$] $\varphi_k:C_k \longrightarrow \{b_{v_1}>b_{v_2}> \ldots >b_{v_m}\}$, defined by $\varphi_k^{-1}(b_{v_s})=C_k^{v_s}$ for $s\in [m]$, is a poset map.
\end{enumerate}
\end{prop}
\begin{proof}
\begin{enumerate}
\item[$(i)$]  Consider $\si \in C_k$. From Equation \eqref{eq42}, $\si \cap V_{u}^k\neq \emptyset$. Thus, $\si \in C_k^v$ for some $v \in V_{u}^k$. Since $C_k^v \subseteq C_k$, $\forall \ v \in V_{u}^k$, we see that $\bigsqcup\limits_{v \in V_{u}^k} C_{k}^v \subseteq C_k$.

\item[$(ii)$] Let $\si, \tau \in C_k$ and $\tau \subseteq \si$. Let $\tau \in C_k^{v_s}$, where $v_s \in V_{u}^k$. Here, $v_s \in \tau$ implies that $v_s \in \si$. Therefore, $\varphi_k(\si) \geq \varphi_k(\tau)$ implying that $\varphi_k$ is a poset map.
\end{enumerate}
\end{proof}

Since, we want to define a perfect acyclic matching on $C_k$, using Lemma \ref{lem34} and Proposition \ref{prop43}, we observe that it is sufficient to define one on $C_k^v$ for each $v \in V_{u}^k$. But, from
Proposition \ref{prop2.1}$(i)$ and $(ii)$, we observe that for this, it is sufficient to define a perfect acyclic matching on $C_k^{12\ell}$ for $\ell \in [k+5]\sm \{1,2\}$.  In this direction, we first define the following:

\begin{Def}
Let $v=s_1s_2s_3 \in V(KG_{3,k})$, where $1\leq s_1<s_2<s_3\leq k+6$. If $\si \in C_k$, define
\begin{equation}\label{eq4.11}
\begin{split}
Cover(v) & =\{t \in [k+6] \ | \ s_1 \leq t \leq s_3\}, \ Co(\si)  =\bigcup\limits_{u \in \si}Cover(u), \\
Cover(\si) & = \{t \ | \ 1 \leq min(Co(\si))\leq t \leq max(Co(\si))\leq k+6\}, \ L(\si)  = \#Cover(\si), \\
Comp(v,\ell) & =\{t \in Cover(v) \ | \ |t-\ell|>1, t \neq s_1,s_2,s_3\}   \text{ and } R(v,\ell) = min\{Comp(v,\ell)\}.
\end{split} 
\end{equation}
\end{Def}

{\bf Example.} In $S_{3,4}$, $Cover(357)=\{3,4,5,6,7\}, \ Co(\{234,679\})=\{2,3,4,6,7,8,9\}, \\ Cover(\{234,679\})=[9]\sm \{1\}, L(\{124,678\})=8, Comp(359,8)=\{4,6\}$ and $R(359,8)=4$. 
If $v=12\ell,$ for any $\si \in C_k^v$, from Equation \eqref{eq42}, $N(\si)\subseteq V(SG_{3,k})$, {\it i.e.}, $L(\si)\geq 5$.
Further, $1,2 \nin Cover(N(\si))$ implies that $L(N(\si)) \leq k+4$.

Let
$C_{k,n}^{12\ell} = \{\si \in C_k^{12\ell} \ | \  L(N(\si))=n \}, \ n \in \{5,\ldots, k+4\}, \ \ell \in \{3,\ldots,k+5\}.$

Here, $C_k^{12\ell}=\bigsqcup\limits_{5}^{k+4}C_{k,n}^{12\ell} $.
Since, $\tau \subseteq \si $ in $C_k^{12\ell}$ implies that $L(N(\si))\leq L(N(\tau))$, the map $\psi_{k}^{\ell} : C_k^{12\ell} \longrightarrow \{c_5>c_6>\ldots >c_{k+4}\} $ defined by $(\psi_{k}^{\ell})^{-1}(c_n)=C_{k,n}^{12\ell}$ is a poset map.
Therefore, to construct an acyclic matching on $C_k^{12\ell}$, it is sufficient to define one on $C_{k,n}^{12\ell}$ for each $n \in \{5,\ldots,k+4\}$.
To get this matching first define $\B_{n}^{\ell} \subseteq V(SG_{3,k})$ as 
\begin{equation*}
\B_{n}^{\ell}=\{u \in N(12\ell) \ | \ u \in N(\si) \text{ for some } \si \in C_{k,n}^{12\ell} \}.
\end{equation*}

\begin{claim}\label{claim4.5}
$Comp(u, \ell)\neq \emptyset$, $\forall \ u \in \B_{n}^{\ell}$.
\end{claim} 

Let $u=s_1s_2s_3 \in \B_{n}^{\ell}$, where $s_1<s_2<s_3$. Here, $\B_{n}^{\ell} \subseteq V(SG_{3,k})$ implies that $s_1+1<s_2<s_3-1$. Further, $s_1+1,s_3-1 \in Cover(u)$. If $|s_1+1-\ell| >1$, then $s_1+1 \in Comp(u,\ell)$. If $|s_1+1-\ell| \leq 1$, then clearly $|s_3-1-\ell|>1$. Therefore, $s_3-1 \in Comp(u,\ell)$.

Consider $\B_{n}^\ell=\{u_1<u_2< \ldots <u_{t}\}$. Using  Claim \ref{claim4.5} and $\B_{n}^{\ell} \subseteq V(SG_{3,k})$, we get $R(u_i,\ell) \in \{3,\ldots,k+5\}$. Moreover, $|R(u_i,\ell)-\ell|>1, \ \forall \ i \in \{1,\ldots,t\}$. Therefore, $1\ell R(u_i,\ell) \in V(SG_{3,k})$. 

Define a matching $M_{k,n}^\ell$ on $C_{k,n}^{12\ell}=(\psi_{k}^{\ell})^{-1}(c_n)$ by 
\[
  \sigma \longrightarrow \left\{\def\arraystretch{1.2}%
  \begin{array}{@{}c@{\quad}l@{}}
    \sigma \cup 1\ell R(u_1,\ell) & \text{if } u_1 \in N(\si), \ 1\ell R(u_1,\ell) \nin \sigma,\\
    \sigma \cup 1\ell R(u_2,\ell) & \text{if } u_2 \in N(\si), \ u_1 \notin N(\si), \ 1\ell R(u_2,\ell) \nin \sigma,\\

    \vdots \\
    \sigma \cup 1\ell R(u_t,\ell) & \text{if } u_t \in N(\si),\ u_s \notin N(\si) \ \forall \ 1 \leq s < t, \ 1\ell R(u_t,\ell) \nin \si.\\
  \end{array}\right.
\]

\begin{lem}\label{lem4.3} 
$M_{k,n}^\ell$ is a perfect acyclic matching on $C_{k,n}^{12\ell}$.
\end{lem}
\begin{proof}
Let $\sigma \in C_{k,n}^{12\ell}$ and $\B_{n}^\ell=\{u_1<u_2< \ldots <u_{t}\}$. Define 
\begin{equation*}
r_\si=min\{j \in \{1,2,\ldots,t\}\ | \ u_j \in N(\si)\}. 
\end{equation*}

Since $12\ell \in \si$ and $\si \in \N(S_{3,k})$, $N(\si) \subseteq N(12\ell)$. Hence, there exists $u \in \B_{n}^\ell$ such that $u \in N(\sigma)$. Therefore, $r_\si$ exists.

Further, $R(u_{r_\si}, \ell) \in Comp(u_{r_\si},\ell)$ implies that, if $1\ell R(u_{r_\si},\ell) \nin \si$ then $u_{r_\si} \in N(\si \cup 1\ell R(u_{r_\si},\ell)$ and $L(N(\si))= L(N(\si \cup 1\ell R(u_{r_\si},\ell))=n$. Therefore, if $1\ell R(u_{r_\si},\ell)) \nin \si$ then $ \si \cup 1\ell R(u_{r_\si},\ell) \in C_{k,n}^{12\ell}$. 

Consider, $1\ell R(u_{r_\si},\ell) \in \si$. Since $12\ell \in \si$, $C_{\si \sm 1\ell R(u_{r_\si},\ell))}=C_\si \text{ or } C_\si \cup \{R(u_{r_\si},\ell)\}$. Here, $R(u_{r_\si}, \ell) \in Comp(u_{r_\si},\ell)$ and $u_{r_\si} \in N(\si \sm 1\ell R(u_{r_\si},\ell))$. Thus, $L(N(\si \sm 1\ell R(u_{r_\si},\ell)))=L(N(\si))$. Therefore, $\si \sm 1\ell R(u_{r_\si},\ell) \in C_{k,n}^{12\ell}.$ This proves Lemma \ref{lem4.3}. 
\end{proof}

\begin{proof}[Proof of Theorem \ref{thm2}]

From Lemmas \ref{lemma21}, \ref{lem4.2} and \ref{lem4.3}, we have perfect acyclic matchings $M_1$, $M_2$ and $M_3$ on $\varphi_k^{-1}(a_k^1), \ \varphi_k^{-1}(a_k^2)$ and $\varphi_k^{-1}(a_k^3),$ repectively. Using Lemma \ref{lem34} and Proposition \ref{prop42}$(iii)$, $M_1\sqcup M_2\sqcup M_3$ is an acyclic matching on $\mathcal{F}(\N(S_{3,k}))$ with $\varphi_k^{-1}(a_k^4)=\mathcal{F}(\N(SG_{3,k}))$ (Equation \eqref{eqn43})  as the set of critical cells.  Theorem \ref{thm2} now follows from Corollary \ref{rem37}.
\end{proof}

\section{Proof of Theorem \ref{thm3}}
In this section, ${N}(\si)$ will denote the neighborhood of $\si$ in $ KG_{3,k}$ and $\F(X)$ will represent the face poset of the simplicial complex $X$. Let
\begin{equation}\label{eq5.1}
\begin{split}
X_0^{k} & =\N(SG_{3,k}), \ \P_{0}^k=\F(X_0^k), \\
X_1^{k} & =\N(S_{3,k}), \ \P_{1}^k=\F(X_1^k), \\
X_2^{k} & =X_1^k\bigsqcup \{\si \in X_3^k \sm X_1^k \ | \ |C_\si|=4\}, \ \P_{2}^k=\F(X_2^k) \text{ and} \\
X_3^{k} & =\N(KG_{3,k}), \ \P_{3}^k=\F(X_3^k).\\
\end{split}
\end{equation}
It is easy to see that
$X_2^k$ is a subcomplex of $X_3^k$.
Our aim now is to suitably partition $\P_3^k\sm \P_2^k$ and $\P_2^k\sm \P_1^k$ and define acyclic matchings on them. Let $I_i$ and $J_i$ be as defined in Equation \eqref{eq41}. Define 
\begin{equation}\label{eq5.2}
\begin{split}
\P_{i,j}^k & =\{\si \in \P_3^k \sm \P_2^k\ | \ C_\si=\{i,i+1,j\} \} \text{ for } i \in [k+6], j \in I_{i}. \\ 
\Q_{i,j}^k & =\{\si \in \P_2^k \sm \P_1^k\ | \ C_\si=\{i,i+1,j,j+1\} \} \text{ for } i \in [k+4], j \in J_{i}.
\end{split}
\end{equation}

Let $\si \in \P_3^k \sm \P_1^k$. Since $\si \nin \N(S_{3,k})$, we observe that $\si \cap V_{u}^k \neq \emptyset$ and $N(\si)\subset V_{u}^k$. Here, $N(\si)\cap V(SG_{3,k})=\emptyset$ shows that $3 \leq |C_\si| \leq 4$. If $|C_\si|=3$, then there exists $i \in [k+6]$, $j \in I_{i}$ such that $C_\si=\{i,i+1,j\}$, implying that $\si \in \bigsqcup\limits_{i\in [k+6] }\bigsqcup\limits_{j\in I_i} \P_{i,j}^k$. In the case that $|C_\si|=4$, there exists  $i \in [k+4]$, $j \in J_{i}$ such that $C_\si=\{i,i+1,j,j+1\}$. Thus, $\si \in \bigsqcup\limits_{i \in [k+4] } \bigsqcup\limits_{j\in J_i} \Q_{i,j}^k$. Therefore, we get the following partition of $\P_3^k$.

\begin{prop}\label{prop51}
$\P_3^k=\P_1^k \ \bigsqcup \ ( \bigsqcup\limits_{i\in [k+6] }\bigsqcup\limits_{j\in I_i} \P_{i,j}^k)\ \bigsqcup \ ( \bigsqcup\limits_{i \in [k+4] } \bigsqcup\limits_{j\in J_i} \Q_{i,j}^k)$.
\end{prop}

Henceforth, each vertex of $V_{u}^k$ is considered as follows: $v=s_1s_2s_3\in V_{u}^k$ implies that $1\leq s_1<s_2<s_3 \leq k+6$. We now construct an acyclic matching on $\bigsqcup\limits_{i\in [k+6] }\bigsqcup\limits_{j\in I_i} \P_{i,j}^k$.

\begin{lem}\label{lem5.1} Let $V_{u}^k=\{v_1<v_2<\ldots <v_m\}$. For $i \in [k+6]$, the following are poset maps:
\begin{enumerate}
  \item[(i)] $ \varphi : \P_3^k \longrightarrow \{a_1  >  a_2 >  a_3 > \ldots > a_{k+6} > a\}$ defined by 
  \[
  \varphi^{-1}(x) = \left\{\def\arraystretch{1.2}%
  \begin{array}{@{}c@{\quad}l@{}}
    \bigsqcup\limits_{j\in I_i} \P_{i,j}^k& \text{if } x=a_{i}, \ i \in [k+6],\\
    \P^k_2& \text{if $x=a$.}
  \end{array}\right.
\]
  
    \item[(ii)] $\varphi_i:\bigsqcup\limits_{j\in I_i} \P_{i,j}^k \longrightarrow \{c_{i+2}  >  c_{i+3} > \ldots > c_{k+6} > c_1>\ldots>c_{i-3}>c_{i-2}\}$ where $\varphi_i^{-1}(c_j)= \P_{i,j}^k $.
    
  \item[(iii)] $\varphi_{i,j}:\P_{i,j}^k \longrightarrow \{b_{v_1}>b_{v_2}> \ldots >b_{v_m}\}$, defined by 
\begin{equation*}
\varphi_{i,j}^{-1}(b_{v_s})=\{\si \in \P_{i,j}^k,  v_s \in \si \ | \ v_t \nin \si\text{ for any } t<s\leq m \}.
\end{equation*}

  \end{enumerate}
\end{lem}

\begin{proof}
\begin{enumerate}
\item[$(i)$] Let $\tau \subseteq \si$. If $\tau \in \bigsqcup\limits_{j\in I_i} \P_{i,j}^k$ then $C_\tau =\{i,i+1,j\}$. Here, $ C_\si \subseteq C_\tau$ implies that $C_\si =\{i,i+1,j\}$. Therefore, $ \si \in \bigsqcup\limits_{j\in I_i} \P_{i,j}^k$, proving $(i)$.

\item[$(ii)$] Let $\tau \subseteq \si$ and $\tau \in \P_{i,j}^k $. Here, $ C_\tau=\{i,i+1,j\}$. Therefore, $C_\si =\{i,i+1,j\}$, {\it i.e.}, $ \si \in \P_{i,j}^k$ and $\varphi_i(\si) = \varphi_i(\tau)$.

\item[$(iii)$]For all $\si \in \P_{i,j}^k$, $\si \nin \N(S_{3,k})$, {\it i.e.}, $\si \cap V_{u}^k\neq \emptyset$. Therefore, $\varphi_{i,j}$ is well defined.

Let, $\tau \subseteq \si$ and $\tau \in \varphi_{i,j}^{-1}(b_{v_s})$. Here, $\tau \in \varphi_{i,j}^{-1}(b_{v_s})$ implies that $v_s \in \tau$ and $v_t \nin \tau$ for any $t<s$. Further, $\tau \subseteq \si$ implies that $ v_s\in \si.$ Therefore, $\si \in \varphi_{i,j}^{-1}(b_{v_s})$ or $\si \in \varphi_{i,j}^{-1}(b_{v_t})$ for some $v_t \in V_{u}^k, \text{ with } t<s$, implying that $\varphi_{i,j}$ is a poset map.
\end{enumerate}
\end{proof}

From Lemmas \ref{lem34} and \ref{lem5.1}$(iii)$, to construct an acyclic partial matching on $\P_{i,j}^k$, it is sufficient to construct one on $\varphi_{i,j}^{-1}(b_{v_s})$ for each $v_s \in V_{u}^k$. Therefore, using Proposition \ref{prop2.1}$(ii)$, it is sufficient to define an acyclic matching on $\varphi_{1,j}^{-1}(b_{v_s})$ for each $j \in I_1=[k+5]\sm \{1,2\}$.

\subsection{Acyclic matching $M_{1,j,v_s}^k$ on $\varphi_{1,j}^{-1}(b_{v_s})$ for $j \in I_1$.}
To construct an acyclic partial matching on $\P_{1,j}^k$, from Lemma \ref{lem5.1}$(iii)$, it is sufficient to do so on $\varphi_{1,j}^{-1}(b_{v_s})$. For this, define
\begin{equation}\label{eq5.5}
\begin{split}
NC_j & =\big{\{}q(s-1)s\in V_{u}^k \ | \  q \in [k+4]\sm\{1,2,3,j,j+1\},  \\
& \hspace*{5.2cm} s\in [k+6]\sm \{[q+1]\cup \{j,j+1\}\}\big{\}}\text{ and} \\ C_j & =\{v \in V_{u}^k\sm NC_j \ | \ v \cap \{1,2,j\} = \emptyset \}.
\end{split}
\end{equation}

Let $v_s=s_1s_2s_3 \in V_{u}^k$ with $1\leq s_1<s_2<s_3\leq k+6$. Clearly, $\{s_1,s_2,s_3\} \cap\{1,2,j\}\neq \emptyset$ implies that $\varphi_{1,j}^{-1}(b_{v_s})= \emptyset$. Further, if $s_3=s_2+1$ and $s_1 \nin \{1,2,3,j,j+1\}$, then $v_s \in NC_j$. Therefore, if $v_s\in C_j$ and $s_1 \nin \{3,j+1\}$, then $s_3 \neq s_2+1$. Now, define

\begin{equation}\label{eq5.6}
  W_{v_s}^j = 
  \begin{cases*}
    \big{\{}t_1s_2s_3 \ | \ t_1 \in [k+6]\sm \{[s_1] \cup \{j,s_2,s_3\}\}\big{\}} & if $v_s \in NC_j$, \\ 
    \hspace*{0.5cm}\bigsqcup\{t_2s_1s_3\ | \ t_2 \in [s_1-2]\sm \{1,2,j\}\} &  \vspace{0.3cm}\\   
   \big{\{}ts_2(s_2+1) \ | \ t \in [k+6]\sm \{1,2,3,j,s_2,s_2+1\}\big{\}} & if $v_s\in C_j, j\neq 3$,\\
   & $v_s=3s_2(s_2+1)$, \vspace{0.3cm}\\
   \big{\{}t_1s_2(s_2+1) \ | \ t_1 \in [k+6]\sm \{[j+1]\cup\{s_2,s_2+1\}\}\big{\}} & if $v_s\in C_j, v_s=$ \\
  \hspace*{0.5cm} \bigsqcup \{t_2(j+1)(s_2+1) \ | \ t_2\in [j-1] \sm \{1,2\}\} 
  & $(j+1)s_2(s_2+1)$, \vspace{0.3cm}\\
    \{ts_2s_3 \ | \ t \in [k+6]\sm \{1,2,j,s_1,s_2,s_3\}\}  & if $v_s\in C_j$,\\    
    & $s_3 \neq s_2+1$, \vspace{0.3cm} \\    
    \emptyset & otherwise.
  \end{cases*}
\end{equation}

If $(\Sigma)_v$ is as defined in Equation \eqref{eq3.2}, for simplicity, we denote 
\begin{enumerate}
\item $\varphi_{1,j}^{-1}(b_{v_s})$ by $\Sigma^{v_s,j}$, $\Sigma^{v_s,j}\sm (\Sigma^{v_s,j})_{v_1}$ by $\Sigma^{v_s,j}_{\{v_1\}}$ and 
\item $ \Sigma^{v_s,j}_{\{v_1,v_2,\ldots, v_{e-1}\}}\sm(\Sigma^{v_s,j}_{\{v_1,v_2,\ldots, v_{e-1}\}})_{v_e}$ by $\Sigma^{v_s,j}_{\{v_1,v_2,\ldots, v_e\}}$, where $v_1<v_2<\ldots<v_e$. 
\end{enumerate}

If $v_t<v_s$, then from Lemma \ref{lem5.1}$(iii)$, $\si\cup v_t , \ \si \sm v_t \nin \Sigma^{v_s,j}$. Therefore, $(\Sigma^{v_s,j})_{v_t}=\emptyset$ and $(Q)_{v_t}=\emptyset$, for any $Q\subseteq \Sigma^{v_s,j}$.
To construct an acyclic matching on $\varphi_{1,j}^{-1}(b_{v_s})$ for each $j \in I_1=[k+5]\sm\{1,2\}$, we prove some preliminary results.
\begin{prop}\label{prop5.3} 
Let $v_s=s_1s_2s_3 \in V_{u}^k$ and $1\leq s_1<s_2<s_3 \leq k+6$.
 \begin{enumerate}
\item[$(a)$] If $v_t \in W_{v_s}^j \cap V_{u}^k$, then $v_s<v_t$.

\item[$(b)$] $\Sigma^{v_s,j}_{W_{v_s}^j}$ has exactly one $k$ cell if $v_s \in C_j$ and is empty otherwise.

\item[$(c)$] Let $W_{v_s}^j=\{w_1^j,\ldots,w_{p_j}^j\}$, where $w_1^j <\ldots < w_{p_j}^j$ in the lexicographic ordering of triples. The map $\psi_j: \Sigma^{v_s,j} \longrightarrow \{b_1>b_2>\ldots >b_{p_j}>b\}$ defined below is a poset map.
\begin{equation*}
\psi_j^{-1}(x) = 
\begin{cases*} 
    (\Sigma^{v_s,j}_{\{w^j_1,\ldots,w^j_{i-1}\}})_{w^j_i} & if $x=b_i, \ i \in [p_j]$,\\
    \Sigma^{v_s,j}_{W_{v_s}^j}& if $x=b$,\\
\end{cases*}
\end{equation*}
\end{enumerate}
\end{prop}

\begin{proof}
If $v_s \nin NC_j\cup C_j$, then $W_{v_s}^j=\emptyset$ (Equation \eqref{eq5.6}). Consider $v_s \in NC_j\cup C_j$.
\begin{enumerate}
\item[$(a)$] First, let $v_s \in NC_j$, {\it i.e.}, $s_3=s_2+1$ and $s_1 \in [k+4]\sm \{1,2,3,j,j+1\}$.

\begin{claim}\label{claim5.4}
$v_t \in W_{v_s}^j \cap V_{u}^k \rar v_t=t_1s_2(s_2+1)$ where $ t_1 \in [k+6]\sm \{[s_1] \cup \{j,s_2,s_3\}\}$.
\end{claim}
Since, $v_t \in W_{v_s}^j$, it satisfies the first condition of Equation \eqref{eq5.6}. Suppose $v_t=t_2s_1s_3$. From Equation \eqref{eq5.6}, $t_2 \in [s_1-2]\sm \{1,2,j\}$. Here, $3\leq t_2 < s_1-1 <s_2$, {\it i.e.}, $s_1-t_2>1$ and $s_3-t_2>1$. Further, $s_1<s_2$ implies that $s_3-s_1 >1$. Thus, $v_t \in V(SG_{3,k})$, an impossibility. Claim \ref{claim5.4} is thus proved. 

From Claim \ref{claim5.4} and Equation \eqref{eq5.6}, $t_1 \nin [s_1]$ implies that $s_1<t_1$, {\it i.e.} $v_s<v_t$.
Now, let $v_s \in C_j$. Consider the case when $s_3=s_2+1$. 

If $v_s=3s_2(s_2+1)$, {\it i.e.}, $j\neq 3$, then from Equation \eqref{eq5.6}, $v_t \in W_{v_s}^j$. Thus, $v_t=t_1s_2(s_2+1)$, $t_1 >3$ which implies that, $v_s<v_t$.

If $v_s=(j+1)s_2(s_2+1)$, then $v_t \in W_{v_s}^j$ implies that, either $v_t=t_1s_2(s_2+1)$, with $t_1 >j+1$ or $v_t=t_2(j+1)(s_2+1)$, with $t_2<j-1$. 

If $v_t=t_2(j+1)(s_2+1)$ with $t_2<j-1$, then $(s_2+1)-t_2>1$ and $t_2-(j+1)>1$, thereby implying that $v_t \in V(SG_{3,k})$, a contradiction. Therefore, $v_t=t_1s_2(s_2+1)$ with $t_1 >j+1$. Since $t_1 >j+1$, $v_s<v_t$.

Now, consider $v_s \in C_j$, where $s_3\neq s_2+1$. Here, $s_2=s_1+1$. Since $v_t \in W_{v_s}^j$, from Equation \eqref{eq5.6}, $v_t=t_1s_2s_3$ where $t_1 \in [k+6] \sm \{1,2,j,s_1,s_2,s_3\}$. If $t_1<s_1$, then $s_2-t_1>1$, {\it i.e.}, $v_t \in V(SG_{3,k})$ (since $s_3-s_2>1$ and $s_1<s_2<s_3$). Therefore, $s_1<t_1$ and hence $v_s<v_t$.

\item[$(b)$] Let $W_{v_s}^j=\{w^j_1,\ldots,w^j_{p_j}\}\subseteq V_{u}^k$ have the lexicographic ordering. By Equation \eqref{eq5.6}, $w_i^j \cap \{1,2,j\}=\emptyset, \ \forall \ i \in [p_j].$ 

\begin{claim}\label{claim5.5}
\begin{enumerate}
\item[$(i)$] $\Sigma^{v_s,j}_{\{w_1^j\}}=\{\si \in \Sigma^{v_s,j} \ | \ C_{\si \sm w_1^j}\neq \{1,2,j\}\}$.
\item[$(ii)$] $\Sigma^{v_s,j}_{\{w_1^j,\ldots,w_i^j\}}=\{\si \in \Sigma^{v_s,j}_{\{w_1^j,\ldots,w_{i-1}^j\}} \ | \ C_{\si \sm w_i^j}\neq \{1,2,j\}\}, \ i \in [p_j]$.
\end{enumerate}
\end{claim}
Let $\si \in \Sigma^{v_s,j}=\varphi_{1,j}^{-1}(b_{v_s})$. From Lemma \ref{lem5.1}$(iii)$, $\si \in \P_{1,j}^k \ {\it i.e.}, C_\si =\{1,2,j\}$.

First, let $w_1^j \in \si$. If $C_{\si \sm w_1^j}=\{1,2,j\}$, then  $\si, \si \sm w_1^j \in \Sigma^{v_s,j}$ (from Lemma \ref{lem5.1}$(iii)$), {\it i.e.}, $\si \in (\Sigma^{v_s,j})_{w_1^j}$. Here, $w_1^j \in \si$ and $C_{\si \sm w_1^j}\neq \{1,2,j\} $ imply that $ \si \in \Sigma^{v_s,j}_{\{w_1^j\}}.$

If $w_1^j \nin \si$, then $C_{\si \cup w_1^j}=\{1,2,j\}$ (since $w_1^j \cap \{1,2,j\}=\emptyset$). From Proposition \ref{prop5.3}$(a)$, $v_s<w_1^j$. Therefore, $\si, \si\cup w_1^j \in  \Sigma^{v_s,j}$, {\it i.e.}, $\si \in (\Sigma^{v_s,j})_{w_1^j}$. This proves Claim \ref{claim5.5}$(i)$.  

Claim \ref{claim5.5}$(ii)$ is proved by induction. The base case, {\it i.e.,} $i=1$, follows from Claim \ref{claim5.5}$(i)$. Suppose that Claim \ref{claim5.5}$(ii)$ is true for $i=t\in [p_j-1]$. 

Let $i=t+1$ and $\si \in \Sigma^{v_s,j}_{\{w_1^j,\ldots,w_{t}^j\}}$, {\it i.e.}, $C_{\si \sm w_\ell^j} \neq C_\si$ for all $\ell \in [t]$. From Lemma \ref{lem5.1}$(iii)$, $\si \in \P_{1,j}^k, {\it i.e.}, C_\si =\{1,2,j\}$. 

Consider $w_{t+1}^j \in \si$ and $\ell \in [t]$. Since $C_{\si \sm w_\ell^j}\neq \{1,2,j\}$, $C_{\{\si\sm w_{t+1}^j\}\sm w_\ell^j} \neq \{1,2,j\}$. If $C_{\si \sm w_{t+1}^j}=\{1,2,j\}$, then  $\si, \si \sm w_{t+1}^j \in \Sigma^{v_s,j}_{\{w_1^j,\ldots,w_{t}^j\}}$.  Thus, $\si \in (\Sigma^{v_s,j}_{\{w_1^j,\ldots,w_{t}^j\}})_{w_{t+1}^j}$. Further, if $C_{\si \sm w_{t+1}^j}\neq \{1,2,j\}$, then $\si \sm  w_{t+1}^j\nin \P_{1,j}^k$ implying that $\si \nin (\Sigma^{v_s,j}_{\{w_1^j,\ldots,w_{t}^j\}})_{w_{t+1}^j}$, {\it i.e.}, $\si \in \Sigma^{v_s,j}_{\{w_1^j,\ldots,w_{t+1}^j\}}.$

Now, consider the case $w_{t+1}^j \nin \si$. For $\ell \in [t+1]$, let $w_\ell^j = s_{i_1}^\ell s_{i_2}^\ell s_{i_3}^\ell$. From Equation \eqref{eq5.6}, $\# (\{s_{i_1}^\ell ,s_{i_2}^\ell, s_{i_3}^\ell \}\cap \{s_1,s_2,s_3\}) =2$. Without loss of generality, let $s_{i_1}^\ell \nin \{s_1,s_2,s_3\}$, $\forall \ \ell \in [t+1]$. Since $\si \in \Sigma^{v_s,j}_{\{w_1^j,\ldots,w_{t}^j\}}, C_\si=\{1,2,j\}, C_{\si \sm w_\ell^j} \neq \{1,2,j\}$ and $v_s \in \si$, we get $C_{\si \sm w_\ell^j}= \{1,2,j,s_{i_1}^{\ell}\}$ for all $\ell \in [t]$. From Equation \eqref{eq5.6}, $s_{i_1}^{\ell} \neq s_{i_1}^{\ell^\prime}$, $\forall \ \ell \neq \ell^\prime$. Therefore, $C_{\{\si \cup w_{t+1}^j\}\sm w_\ell^j}= \{1,2,j,s_{i_1}^{\ell}\}$. Thus, $\si, \si\cup w_{t+1}^j \in  \Sigma^{v_s,j}_{\{w_1^j,\ldots,w_{t}^j\}}$, {\it i.e.}, $\si \in (\Sigma^{v_s,j}_{\{w_1^j,\ldots,w_{t}^j\}})_{w_{t+1}^j}$. This proves Claim \ref{claim5.5}$(ii)$.

To prove Propostition \ref{prop5.3}$(b)$, first let $v_s=s_1s_2s_3 \in NC_j$, {\it i.e.}, $s_1\neq 1,2,3,j,j+1$ and $s_3=s_2+1$. From Claim \ref{claim5.5}$(ii)$ and Equation \eqref{eq5.6}, 
\begin{equation*}
\begin{split}
 \Sigma^{v_s,j}_{W_{v_s}^j} & = \big{\{}\si \in \Sigma^{v_s,j} \ | \ C_{\si \sm \{t_1s_2s_3\}}  =\{1,2,j,t_1\}, \ \forall \ t_1 \in [k+6] \sm \{[s_1] \cup \{j,s_2,s_3\}\}, \\
& C_{\si \sm \{t_2s_1s_3\}}=\{1,2,j,t_2\}, \ \forall \ t_2 \in [s_1-2]\sm \{1,2,j\}\big{\}} .
 \end{split}
 \end{equation*}
Let $\si \in \Sigma^{v_s,j}_{W_{v_s}^j}$. Here, $\si \in \varphi_{1,j}^{-1}(b_{v_s})$. From Lemma \ref{lem5.1}$(iii)$, $x_1(x_1+1)x_3, \ x_1x_2(x_2+1)\nin \si,$ $\forall \ x_1 <s_1$. Further, from Equation \eqref{eq5.6}, $(s_1-1)s_2(s_2+1), (s_1-1)s_1(s_2+1) \nin W_{v_s}^3$. Therefore, $(s_1-1)s_2(s_2+1), (s_1-1)s_1(s_2+1) \nin  \si$ implying that $s_1-1 \nin S_\si$, {\it i.e.}, $s_1-1 \in C_\si$. Further, $s_1-1 \nin \{1,2,j\}$ implies that $\si \nin \varphi_{1,j}^{-1}(b_{v_s})$, a contradiction. Therefore, if $v_s \in NC_j$, then $\Sigma^{v_s,j}_{W_{v_s}^j}=\emptyset.$
In the case, $v_s \in C_j$, there are three possibilities for $v_s$.  

If $v_s=3s_2(s_2+1)$, from Claim \ref{claim5.5}$(ii)$ and Equation \eqref{eq5.6} we get, 
\begin{equation*}
\begin{split}
\Sigma^{3s_2(s_2+1),j}_{W_{3s_2(s_2+1)}^j} & = \big{\{}\si \in \Sigma^{3s_2(s_2+1),j} \ | \ C_{\si \sm \{ts_2(s_2+1)\}}=\{1,2,j,t\}, \ \forall \ t \neq 1,2,3,j,s_2,s_2+1\big{\}} \\
& =\big{\{}\{ts_2(s_2+1)\ | \ t \in [k+6] \sm \{1,2,j,s_2,s_2+1\}\}\big{\}}.
\end{split}
\end{equation*} 

If $v_s=s_1(s_1+1)s_3$ with $s_3>s_1+2$ then,
\begin{equation*}
\begin{split}
\Sigma^{v_s,j}_{W_{v_s}^j} & = \big{\{}\si \in \Sigma^{v_s,j} \ | \ C_{\si \sm \{t(s_1+1)s_3\}} =\{1,2,j,t\}, \ \forall \ t \neq 1,2,j,s_1,s_2,s_3\big{\}}\\
 & =\big{\{}\{t(s_1+1)s_3 \ | \ t \in [k+6]\sm \{1,2,j,s_2,s_3\} \}\big{\}}.
 \end{split}
\end{equation*} 

Now, consider $v_s=(j+1)s_2(s_2+1)$. From Claim \ref{claim5.5}$(ii)$ and Equation \eqref{eq5.6},

\begin{equation*}
\begin{split}
 \Sigma^{(j+1)s_2(s_2+1),j}_{W_{(j+1)s_2(s_2+1)}^j} & = \Big{\{}\si \in \Sigma^{(j+1)s_2(s_2+1),j} \ | \ C_{\si \sm \{t_1s_2(s_2+1)\}}  =\{1,2,j,t_1\} \text{ and }  \\
 & \hspace*{0.7cm} C_{\si \sm \{t_2(j+1)(s_2+1)\}}=\{1,2,j,t_2\}, \ \forall \ t_1 \nin [j+1] \cup \{s_2,s_2+1\},\\
 & \hspace*{0.8cm} t_2 \in [j-1]\sm \{1,2\}\Big{\}} \\
 & =\Big{\{\{}3(j+1)(s_2+1),\ldots,(j-1)(j+1)(s_2+1),(j+1)s_2(s_2+1), \\ 
& \hspace*{0.7cm} (j+2)s_2(s_2+1),\ldots,(s_2-1)s_2(s_2+1),s_2(s_2+1)(s_2+2),\\
& \hspace*{0.7cm} \ldots,s_2(s_2+1)(k+5),s_2(s_2+1)(k+6)\}\Big{\}}.
 \end{split}
 \end{equation*}
Therefore, if $v_s \in C_j$, then $ \Sigma^{v_s,j}_{W_{v_s}^j}$ contains exactly one cell of dimension $k$.

\item[$(c)$] Let $\tau , \si \in \Sigma^{v_s,j}$ such that $\tau \subseteq \si$. Assume that $\tau \in \psi_j^{-1}(b_{t_1}),  \ \si \in \psi_j^{-1}(b_{t_2})$ and $t_2>t_1$. From Claim \ref{claim5.5}$(ii)$, $\si \in (\Sigma^{v_s,j}_{\{w^j_1,w^j_2,\ldots, w^j_{t_2-1}\}})_{w^j_{t_2}}$ implies that $C_{\si \sm w^j_t}\neq \{1,2,j\}$ for all $t<t_2$. In particular, $C_{\si \sm w^j_{t_1}}\neq \{1,2,j\}$. 

If $\tau \in (\Sigma^{v_s,j}_{\{w^j_1,w^j_2,\ldots, w^j_{t_1-1}\}})_{w^j_{t_1}}$, then from Claim \ref{claim5.5}$(ii)$, $C_{\tau \sm w^j_{t_1}}=C_\tau=\{1,2,j\}$. But $C_\si \subseteq C_\tau$ implies that $C_{\si \sm w^j_{t_1}} \subseteq C_{\tau \sm w^j_{t_1}}=C_\tau$, {\it i.e.}, $C_{\si \sm w^j_{t_1}}=\{1,2,j\}$, a contradiction. This completes the proof of Proposition \ref{prop5.3}. 
\end{enumerate}
\end{proof}

\begin{lem}\label{lem5.3} 
There exists an acyclic matching on $\P_{1,j}^k$, $j \in \{3,\ldots,k+5\}$, with exactly $\frac{(k+1)(k+2)}{2}$ critical cells of dimension $k$.
\end{lem} 
\begin{proof}
Let $W_{v_s}^j=\{w^j_1,\ldots,w^j_{p_j}\}$ be ordered lexicographically. From Proposition \ref{prop5.3}$(c)$ and Lemma \ref{lem34}, the acyclic matching on $\Sigma^{v_s,j}$ is $M_{1,j,v_s}^k=\bigsqcup\limits_{t=1 }^{p_j} M(\Sigma^{v_s,j}_{\{w^j_1,w^j_2,\ldots,w^j_{t-1}\}})_{w^j_t} .$ The set of critical cells for this matching is $\Sigma^{v_s,j}_{W_{v_s}^j}$. From Proposition \ref{prop5.3}$(b)$,  $ \#C_j=\frac{(k+1)(k+2)}{2}=\#\big{\{}\bigsqcup\limits_{v_s \in V_{u}^k} \Sigma^{v_s,j}_{W_{v_s}^j}\big{\}}$. Using Lemmas \ref{lem34} and \ref{lem5.1}$(iii)$, $M_{1,j}^k=\bigsqcup\limits_{v_s \in V_{u}^k} M_{1,j,v_s}^k$ is an acyclic matching on $\P_{1,j}^k$ with $\frac{(k+1)(k+2)}{2}$ critical cells of dimension $k$.
\end{proof}

In Lemma \ref{lem5.1}, an acyclic matching on $\P_3^k\sm \P_2^k$ is constructed. To define an acyclic matching on $\P_3^k$, we now define one on $\bigsqcup\limits_{i \in [k+4] } \bigsqcup\limits_{j\in J_i} \Q_{i,j}^k$ (Proposition \ref{prop51}). As in the case of the ${\P_{i,j}^k}{'}$s, we partition the ${\Q_{i,j}^k}{'}$s into smaller sets and then construct acyclic matchings on them.

\begin{lem}\label{lem5.4} For $i \in [k+4]$ and $J_i$ as defined in Equation \eqref{eq41}, the following are poset maps:
\begin{enumerate}
  \item[(i)] $ \varphi^\prime : \P_2^k \longrightarrow \{a_1^\prime  >  a_2^\prime >  a_3^\prime > \ldots > a_{k+4}^\prime > a^\prime \}$ defined by 
  \[
  (\varphi^\prime)^{-1}(x) = \left\{\def\arraystretch{1.2}%
  \begin{array}{@{}c@{\quad}l@{}}
    \bigsqcup\limits_{j \in J_i} \Q_{i,j}^k& \text{if }x=a_{i}^\prime, \ i \in [k+4],\\
    \P^k_1& \text{if $x=a^\prime$.}
  \end{array}\right.
\]
  
    \item[(ii)] $\varphi_i^\prime:\bigsqcup\limits_{j\in J_i} \Q_{i,j}^k \longrightarrow \{c_{t_1}^\prime >c_{t_2}^\prime>\ldots > c_{t_{r_i}}^\prime\}$, where $J_i=\{t_1<t_2<\ldots < t_{r_i}\}$, defined by $(\varphi_i^\prime)^{-1}(c_j^\prime)= \Q_{i,j}^k $.
    
  \item[(iii)] $\varphi_{i,j}^\prime:\Q_{i,j}^k \longrightarrow \{b_{v_1}>b_{v_2}> \ldots >b_{v_m}\}$, where $V_{u}^k=\{v_1<v_2<\ldots <v_m\}$, defined by $(\varphi_{i,j}^\prime)^{-1}(b_{v_s})=\{\si \in \Q_{i,j}^k, v_s \in \si \ | \ v_t \nin \si \ \forall \ t<s \}$ for $s \in [m]$.

  \end{enumerate}
\end{lem}
\begin{proof}
This proof is similar to that of Lemma \ref{lem5.1}.
\end{proof}
 We can now define an acyclic matching on $\Q_{i,j}^k$, $i \in [k+4]$, $j \in J_i$.
\begin{lem}\label{lem5.5} 
For $i \in [k+4]$ and $j \in J_i$, there exists an acyclic matching on $\Q_{i,j}^k$ with exactly $\frac{k(k+1)}{2}$ critical cells of dimension $k-1$.
\end{lem}
\begin{proof}
If $\si \in \Q_{i,j}^k$, then $C_\si=\{i,i+1,j,j+1\}$ and $C_{\si\oplus (k+5-j)}=\{i+k+5-j,i+k+6-j,k+5,k+6\}$. From Proposition \ref{prop2.1}, $\si \cap V_{u}^k\neq \emptyset$ implies that $\si \oplus (k+5-j) \cap V_{u}^k\neq \emptyset$. Thus, $\si \in \Q_{i,j}^k$ if and only if $\si \oplus (k+5-j) \in \P_{i+k+5-j,k+5}^{k-1}$.

From Lemma \ref{lem5.3}, $M_{i+k+5-j,k+5}^{k-1}$ is an acyclic matching on $\P_{i+k+5-j,k+5}^{k-1}$ with exactly $\frac{k(k+1)}{2}$ critical cells of dimension $k-1$. Therefore, $N_{i,j}^k=M_{i+k+5-j,k+5}^{k-1}\ominus (k+5-j)$ is an acyclic matching on $\Q_{i,j}^k$ with $\frac{k(k+1)}{2}$ critical cells of dimension $k-1$.
\end{proof}

Before proving the main result, we compute some homology groups.

\begin{lem}\label{lem5.6}
 
 \begin{enumerate} 
 \item[$(i)$] \[
  H_{p}(X_3^k,X^k_2)= \left\{\def\arraystretch{1.2}%
  \begin{array}{@{}c@{\quad}l@{}}
    \bigoplus\limits_{\frac{(k+1)(k+2)(k+3)(k+6)}{2}} \mathbb{Z} & \text{if $p=k$, }\\
    0 & \text{otherwise.}\\
  \end{array}\right.
\]
 \item[$(ii)$] \[
 H_{p}(X^k_2,X^k_1)= \left\{\def\arraystretch{1.2}%
  \begin{array}{@{}c@{\quad}l@{}}
    \bigoplus\limits_{\frac{k(k+1)(k+3)(k+6)}{4}} \mathbb{Z} & \text{if $p=k-1,$ }\\
    0 & \text{otherwise.}\\
  \end{array}\right.
\]
 
 \end{enumerate}
 \end{lem}
 
 \begin{proof}
Observe that, $ \# \{(i,j)\ | \ i \in [k+6], j \in I_i\}= (k+3)(k+6)$. In Lemma \ref{lem5.3}, we have constructed an acyclic matching on $\P_{1,j}^k=\varphi_1^{-1}(c_j)$ with $\frac{(k+1)(k+2)}{2}$ critical cells of dimension $k$, for each $j \in I_1$. Therefore, using Proposition \ref{prop2.1}, we get an acyclic matching $M_{i,j}^k$ on $\P_{i,j}^k=\varphi_i^{-1}(c_j)$ with $\frac{(k+1)(k+2)}{2}$ critical cells of dimension $k$, for each $ i \in [k+6], \ j \in I_i$. 

Using Lemma \ref{lem5.1}$(i)$ and $(ii)$,  the set of critical cells  associated to the matching  on $\P_3^k$ contains the cells of $\P_2^k$  and an additional $\frac{(k+1)(k+2)(k+3)(k+6)}{2}$ cells of dimension $k$. 
Therefore, from Theorem \ref{lem35}, $X_3^k$ is homotopy equivalent to a cell complex $Y_3^k$, where $Y_3^k= X_2^k \bigcup \big{\{} \alpha_r \ | \ r \in [\frac{(k+1)(k+2)(k+3)(k+6)}{2}]\big{\}} $, with each $\alpha_r$ being a $k$-cell. The quotient space $Y_3^k/X_2^k$ is homotopy equivalent to a wedge of $\frac{(k+1)(k+2)(k+3)(k+6)}{2}$ spheres of dimension $k$. 

Since $(Y_3^k,X_2^k)$ is a good pair, $H_p(Y_3^k/X^k_2) \cong H_p(Y_3^k,X^k_2) \cong H_{p}(X_3^k,X_2^k)$ (Proposition \ref{prop3.1}). This proves Lemma \ref{lem5.6}$(i)$.
 
 It is a simple observation that $\#\{(i,j)\ | \ i \in [k+4], \ j \in J_i\}=\frac{(k+3)(k+6)}{2}$. From Lemmas \ref{lem5.4} and \ref{lem5.5}, the set of critical cells associated to the acyclic matching on $ \P_2^k$ are the cells of $\P_1^k$ and a further $\frac{k(k+1)(k+3)(k+6)}{4}$ cells of dimension $k-1$. By arguments similar to those in $(i)$, Lemma \ref{lem5.6}$(ii)$ is proved.
 \end{proof}
 \subsection{Proof of Theorem \ref{thm3}}
 \begin{proof} Longueville (Proposition $2.4$, page $43$, \cite{Markde12}) proved that $X_3^k  =\N(KG_{3,k}) \simeq \bigvee\limits_{t}S^k$ for some $t$. Therefore,
 \begin{equation}\label{eq5.7}
 H_{p}(X_3^k)\cong \left\{\def\arraystretch{1.2}%
  \begin{array}{@{}c@{\quad}l@{}}
    \bigoplus\limits_{t} \mathbb{Z} & \text{if $p=k$, }\\
    0 & \text{otherwise.}\\
  \end{array}\right.
\end{equation}

 Since $(X_3^k,X_1^k)$ is a good pair, $H_p(X_3^k/X^k_1) \cong H_p(X_3^k,X^k_1)$ (Proposition \ref{prop3.1}). 
 
From Corollary \ref{cor1.2}, $X^k_1 \simeq X^k_0 \simeq S^k$. Thus, using Equation \eqref{eq5.7}, we observe that
 \begin{equation}\label{eq5.12}
 H_{p}(X_3^k,X^k_1)\cong \left\{\def\arraystretch{1.2}%
  \begin{array}{@{}c@{\quad}l@{}}
    \bigoplus\limits_{t-1} \mathbb{Z} & \text{if $p=k$, }\\
    0 & \text{otherwise.}\\
  \end{array}\right.
\end{equation}
 
The triple $X^k_1 \subset X^k_2\subset X_3^k$ gives the long exact sequence:
 \begin{equation}\label{eq5.8}
 \begin{array}{@{}c@{\quad}l@{}}
    0 \longrightarrow H_k(X^k_2,X_1^k)=0 \longrightarrow H_k(X_3^k,X^k_1) \longrightarrow H_k(X_3^k,X^k_2)\longrightarrow H_{k-1}(X^k_2,X_1^k) \\
    \longrightarrow H_{k-1}(X_3^k,X^k_1)=0\longrightarrow H_{k-1}(X_3^k,X^k_2)=0\longrightarrow 0. \\
  \end{array}
 \end{equation}
 
Here, $H_{k-1}(X^k_2,X_1^k)$ is a free  group (Lemma \ref{lem5.6}$(ii)$). Therefore, the short exact sequence in Equation \eqref{eq5.8} splits (using Lemma \ref{lem3.4}). Thus,
 
 \begin{equation*}
H_k(X_3^k,X^k_2) \cong H_k(X_3^k,X^k_1) \bigoplus H_{k-1}(X^k_2,X_1^k). 
 \end{equation*}
 
 From Lemma \ref{lem5.6}$(ii)$,
 \begin{equation}\label{eq5.9}
\bigoplus\limits_{\frac{(k+1)(k+2)(k+3)(k+6)}{2}} \mathbb{Z} \cong  H_k(X_3^k,X^k_1) \bigoplus\Big(\bigoplus\limits_{\frac{k(k+1)(k+3)(k+6)}{4}} \mathbb{Z}\Big).
 \end{equation} 
 
Using Equations \eqref{eq5.12} and \eqref{eq5.9}, we conclude that
\begin{equation*} \frac{(k+1)(k+2)(k+3)(k+6)}{2}=t-1+\frac{k(k+1)(k+3)(k+6)}{4}.
\end{equation*}

Therefore, 
\begin{equation*}
\begin{split}
t & =\frac{(k+1)(k+2)(k+3)(k+6)}{2}-\frac{k(k+1)(k+3)(k+6)}{4}+1 \\
&=\frac{(k+1)(k+3)(k+4)(k+6)}{4}+1.
\end{split}
\end{equation*} 

This completes the proof of Theorem \ref{thm3}.
\end{proof}

\bibliographystyle{alpha}

\end{document}